\documentclass[12pt,oneside,reqno]{amsart}

\usepackage{amsmath,amssymb,amsthm,textcomp}
\usepackage{amsfonts,graphicx}
\usepackage[mathscr]{eucal}
\usepackage{color}
\usepackage{csquotes}
\usepackage[backend=bibtex,%
firstinits=true,%
doi=false,%
isbn=true,%
url=false,%
maxnames=99]{biblatex}%

\AtEveryBibitem{\clearfield{issn}}
\AtEveryCitekey{\clearfield{issn}}
\numberwithin{equation}{section}
\DeclareNameAlias{sortname}{last-first}
\theoremstyle{definition}
\usepackage{mathtools}
\numberwithin{equation}{section}

\newcommand{\ncom}{\newcommand}

\ncom{\beq}{\begin{equation}}
	\ncom{\eeq}{\end{equation}}
\ncom{\bea}{\begin{eqnarray*}}
	\ncom{\eea}{\end{eqnarray*}}
\ncom{\beqa}{\begin{eqnarray}}
	\ncom{\eeqa}{\end{eqnarray}}
\ncom{\nno}{\nonumber}
\ncom{\non}{\nonumber}
\ncom{\ds}{\displaystyle}
\ncom{\half}{\frac{1}{2}}
\ncom{\mbx}{\makebox{.25cm}}
\ncom{\hs}{\mbox{\hspace{.25cm}}}
\ncom{\rar}{\rightarrow}
\ncom{\Rar}{\Rightarrow}
\ncom{\noin}{\noindent}
\ncom{\bc}{\begin{center}}
	\ncom{\ec}{\end{center}}
\ncom{\sz}{\scriptsize}
\ncom{\rf}{\ref}
\ncom{\s}{\sqrt{2}}
\ncom{\sgm}{\sigma}
\ncom{\Sgm}{\Sigma}
\ncom{\psgm}{\sigma^{\prime}}
\ncom{\dt}{\delta}
\ncom{\Dt}{\Delta}
\ncom{\lmd}{\lambda}
\ncom{\Lmd}{\Lambda}
\ncom{\Th}{\Theta}
\ncom{\e}{\eta}
\ncom{\eps}{\epsilon}
\ncom{\pcc}{\stackrel{P}{>}}
\ncom{\lp}{\stackrel{L_{p}}{>}}
\ncom{\dist}{{\rm\,dist}}
\ncom{\sspan}{{\rm\,span}}
\ncom{\re}{{\rm Re\,}}
\ncom{\im}{{\rm Im\,}}
\ncom{\sgn}{{\rm sgn\,}}
\ncom{\ba}{\begin{array}}
	\ncom{\ea}{\end{array}}
\ncom{\hone}{\mbox{\hspace{1em}}}
\ncom{\htwo}{\mbox{\hspace{2em}}}
\ncom{\hthree}{\mbox{\hspace{3em}}}
\ncom{\hfour}{\mbox{\hspace{4em}}}
\ncom{\vone}{\vskip 2ex}
\ncom{\vtwo}{\vskip 4ex}
\ncom{\vonee}{\vskip 1.5ex}
\ncom{\vthree}{\vskip 6ex}
\ncom{\vfour}{\vspace*{8ex}}
\ncom{\norm}{\|\;\;\|}
\ncom{\integ}[4]{\int_{#1}^{#2}\,{#3}\,d{#4}}
\ncom{\vspan}[1]{{{\rm\,span}\{ #1 \}}}
\ncom{\dm}[1]{ {\displaystyle{#1} } }
\ncom{\ri}[1]{{#1} \index{#1}}

\newtheorem{theorem}{\bf Theorem}[section]
\newtheorem{remark}{\bf Remark}[section]
\newtheorem{proposition}{Proposition}[section]

\newtheoremstyle
{remarkstyle}
{}
{11pt}
{}
{}
{\bfseries}
{:}
{     }
{\thmname{#1} \thmnumber{#2} }

\theoremstyle{remarkstyle}

\def\eps{\varepsilon}

\begin{document}
\title{Generalized Counting Process with random \\
	drift and different Brownian clocks}
\author[Mostafizar Khandakar]{Mostafizar Khandakar}
\address{Mostafizar Khandakar, Department of Science and Mathematics, Indian Institute of Information Technology Guwahati, Bongora, Assam 781015, India}
\email{mostafizar@iiitg.ac.in}
\author[Manisha Dhillon]{Manisha Dhillon}
\address{Manisha Dhillon, Department of Mathematics, Indian Institute of Technology Bhilai, Durg, 491002, India.}
\email{manishadh@iitbhilai.ac.in}
\author[Kuldeep Kumar Kataria]{Kuldeep Kumar Kataria}
\address{Kuldeep Kumar Kataria, Department of Mathematics, Indian Institute of Technology Bhilai, Durg, 491002, India.}
\email{kuldeepk@iitbhilai.ac.in}
\date{\today}
\begin{abstract}
In this paper, we introduce drifted versions of the generalized counting process (GCP) with a deterministic drift and a random drift. The composition of stable subordinator with an independent inverse stable subordinator is taken as the random drift. We derive the probability law and its governing fractional differential equations for these drifted versions. Also, we study the GCP time-changed with different Brownian clocks, for example, the Brownian first passage-time with or without drift, elastic Brownian motion, Brownian sojourn time on positive half-line and the Bessel times. For these time-changed processes, we obtain the governing system of differential equation of their state probabilities, probability generating function, {\it etc}. Further, we consider a time-changed GCP where the time-change is done by subordinators linked to incomplete gamma function. Later, we study the fractional integral of GCP and its time-changed variant.
\end{abstract}

\maketitle
\section{Introduction}
Di Crescenzo {\it et al.} (2016) introduced and studied a L\'evy process which performs independently $k$ kinds of jumps of amplitude $1,2,\dots,k$ with positive rates $\lambda_1,\lambda_2,\dots,\lambda_k$, respectively. It can be viewed as a generalization of the Poisson process which is known as the generalized counting process (GCP). Here, we denote it by $\{M(t)\}_{t\geq0}$. In an infinitesimal interval of length $h$, its transition probabilities are given by 
\begin{equation*}
\mathrm{Pr}\{M(t+h)=n+j|M(t)=n\}=\begin{cases}
		1-\Lambda h+o(h), \ j=0,\\
		\lambda_j h+o(h),\ 1\leq j\leq k,\\
		o(h),\ j\geq k+1,
\end{cases}
\end{equation*}	
where $\Lambda=\lambda_1+\lambda_2+\cdots+\lambda_k$ and $o(h)/h\to0 $ as $h\to0$. Its state probabilities $p(n,t)=\mathrm{Pr}\{M(t)=n\}$ solve the following system of differential equations:
\begin{equation}\label{gcpdE}
\frac{\mathrm{d}}{\mathrm{d}t}p(n,t)=-\Lambda p(n,t)+\sum_{j=1}^{n\wedge k}\lambda_j p(n-j,t), \ n\geq0,
\end{equation}
with initial conditions $p(0,0)=1$ and $p(n,0)=0$ for all $n\geq1$. Here, $n\wedge k=\min\{n,k\}$. 
Also, Di Crescenzo {\it et al.} (2016) studied its time-changed variant, namley, the generalized fractional counting process (GFCP) $\{M^{\beta}(t)\}_{t\ge0}$, $0<\beta\leq 1$. It is obtained by time-changing the GCP by an independent inverse
stable subordinator $\{Y_\beta(t)\}_{t\geq0}$, that is, $M^{\beta}(t)=M(Y_\beta(t))$.   

  For $\beta=1$, the GFCP reduces to the GCP $\{M(t)\}_{t\ge0}$. For $k=1$, the GFCP and the GCP reduces to the time fractional Poisson process (TFPP) (see Mainardi (2004), Beghin and Orsingher (2009), Meerschaert {\it et al.} (2011)) and the Poisson process, respectively. Kataria and Khandakar (2022a) studied some additional properties of the GFCP and shown a potential application of the GCP in risk theory. Recently, their martingale characterizations are obtained by Dhillon and Kataria (2024). The TFPP is a renewal process with Mittag-Leffler waiting times (see Laskin (2003), Beghin and Orsingher (2009)). It is important to note that the TFPP and GFCP exhibit the long range dependence (LRD) property (see Biard and
   Saussereau (2014), Kataria and Khandakar (2022a)). It is known that a stochastic model having the LRD
   property has non stationary increments and such stochastic models have applications in several areas, such as internet data traffic modelling, finance, econometrics, hydrology, {\it etc}. For other fractional versions of the Poisson process, such as the space fractional Poisson process and the space-time fractional Poisson process, we refer the reader to Orsingher and Polito (2012).
    
Kataria {\it et al.} (2022) introduced and studied space and space-time fractional versions of the GCP, namely, the generalized space fractional counting process (GSFCP) and generalized space-time fractional counting process (GSTFCP). The GSFCP is defined by time-changing the GCP with an independent stable subordinator $\{D_{\beta}(t)\}_{t\ge0}$, that is, $\{M(D_{\beta}(t))\}_{t\ge0}$, $0<\beta<1$ and the GSTFCP is defined by time-changing the GCP with an independent random composition of stable and inverse stable subordinator, that is, $\{M(D_{\beta}(Y_{\alpha}(t)))\}_{t\ge0}$, $0<\alpha<1$. For $k=1$, the GSFCP and the GSTFCP reduces to the space fractional Poisson process and the space-time fractional Poisson process, respectively. For other time-changed variants of the GCP, GFCP and Poisson process, we refer the reader to Kataria and Khandakar (2022b), Orsingher and Toaldo (2015), {\it etc.}

In this paper, we study drifted version and some time-changed variants of the GCP.  In Section \ref{sec2p}, we give some preliminary results that will be used later. In Section \ref{driftsec}, we consider a drifted version of GCP with deterministic drift, that is,  $\{M(t)+bt\}_{t>0}$, $b\ge0$. We obtain the probability law for this process and its governing differential equation. Then, we study the following  time-changed variant of it:
\begin{equation*}
	M^\alpha_{\gamma,\beta}(t)\coloneqq M(D_\gamma(Y_\beta(t)))+bD_\alpha(Y_\beta(t)), \ \alpha\in (0,1], \ \gamma\in (0,1], \ \beta\in (0,1], \ t>0.
\end{equation*} 
Here, all the component processes are independent from each other. Thus, the process $\{M^\alpha_{\gamma,\beta}(t)\}_{t>0}$ can be viewed as a drifted version of the GSTFCP. For $k=1$, it reduces to the drifted version of  space-time fractional Poisson process (see Beghin and D'Ovidio (2014)). For $b=0$, it reduces to the GSTFCP (see Kataria {\it et al.} (2022)). For $\beta=1$, we obtain a governing differential equation for the density function of its first hitting time process.     

In Section \ref{sec2}, we study various time-changed variants of the GCP, where the time-change components are dependent on Brownian motion $\{B(t)\}_{t\geq0}$,  such as, the elastic Brownian motion, the Brownian sojourn time on
positive half-line, the first-passage time of a Brownian
motion with or without drift and the Bessel process. Several distributional properties such as the probability mass function (pmf), mean, variance, probability generating function (pgf), {\it etc.} are derived for these time-changed processes. The results obtained in this section generalizes and compliment the corresponding results for the time-changed Poisson processes studied by Beghin and Orsingher (2012).

Beghin and Ricciuti (2021) introduced a subordinator $\{G_{\alpha}(t)\}_{t\geq0}$, $0<\alpha\le 1$, whose Laplace exponent is given by $\phi(\eta)=\alpha\gamma(\alpha;\eta)$, $\eta>0$ where $\gamma(\alpha; \eta)$ is the lower-incomplete gamma function. Also, they introduced a generalization of $\{G_{\alpha}(t)\}_{t\geq0}$ whose jumps sizes are greater than or equal to $\epsilon>0$. Here, we denote this generalized subordinator by $\{G_{\alpha}^{(\epsilon)}(t)\}_{t\geq0}$. The subordinator $\{G_{\alpha}(t)\}_{t\geq0}$ has no finite moments of any integer order. To overcome this limitation, they introduced a tempered version of it, denoted by  $\{G_{\alpha,\theta}(t)\}_{t\geq0}$, with the Laplace exponent $\phi_{\alpha,\theta}(\eta)=\alpha\gamma(\alpha;\eta+\theta)-\alpha\gamma(\alpha;\theta)$. Here, $\theta>0$ is the tempering parameter. In Section \ref{sec3}, we consider the GCP time-changed by the processes $\{G_{\alpha}^{(\epsilon)}(t)\}_{t\geq0}$ and $\{G_{\alpha,\theta}(t)\}_{t\geq0}$.  We derive the distributional properties, such as, pmf, mean,
variance, asymptotic behavior of tail probabilities of these time-changed processes. Also, we have shown that the GCP time-changed by $\{G_{\alpha,\theta}(t)\}_{t\geq0}$ exhibits the LRD property.

In Section \ref{sec6},  we introduce the fractional integral of GCP and GFCP.  We obtain the mean, variance and conditional mean of fractional integral of the GCP. The study of fractional integrals of these processes is motivated by the fact that integrals of counting processes have various applications in applied mathematics (see Jerwood (1970), Downtown (1972), Pollett (2003)). A similar study for the Poisson process and TFPP is done by Orsingher and Polito (2013). 

\section{Preliminaries}\label{sec2p}
In this section, we give some preliminary details on Mittag-Leffler function, fractional derivatives and some known processes that will be required later.
\subsection{Some special functions}
The three-parameter Mittag-Leffler function is defined as (see Kilbas {\it et al.} (2006), Eq. (1.9.1))
\begin{equation}\label{mitag}
	E_{\alpha,\beta}^{\gamma}(x)=\frac{1}{\Gamma(\gamma)}\sum_{j=0}^{\infty} \frac{\Gamma(j+\gamma)x^{j}}{j!\Gamma(j\alpha+\beta)},\ x\in\mathbb{R},
\end{equation}
where $\alpha>0$, $\beta>0$ and $\gamma>0$. For $\gamma=1$, it reduces to two-parameter Mittag-Leffler function. Further, for $\gamma=\beta=1$, we get the Mittag-Leffler function. 

 The following Laplace transform will be used (see Kilbas {\it et al.} (2006), Eq. (1.9.13)):
\begin{equation}\label{mi}
	\mathcal{L}\big(t^{\beta-1}E^{\gamma}_{\alpha,\beta}(xt^{\alpha});s\big)=\frac{s^{\alpha\gamma-\beta}}{(s^{\alpha}-x)^{\gamma}},\ x\in\mathbb{R},\  s>|x|^{1/\alpha}.
\end{equation}

 The confluent hypergeometric function is defined as (see Gradshteyn and Ryzhik (2007), p. 1023, Section 9.21)
\begin{equation}\label{confluent}
	_1F_1(\alpha;\gamma;x)=1+\sum_{j=1}^{\infty}\frac{\alpha(\alpha+1)\dots(\alpha+j-1)}{\gamma(\gamma+1)\dots(\gamma+j-1)}\frac{x^j}{j!},\ \gamma\neq -n, \ n= 0,1, \dots.
\end{equation}

The modified Bessel function of index $\alpha$ is defined as (see  Gradshteyn and Ryzhik (2007), p. 917) 
\begin{equation}\label{modifiedbes}
	K_{\alpha}(s)=\frac{1}{2}\bigg(\frac{s}{2}\bigg)^\alpha\int_{0}^{\infty}\frac{e^{-t-\frac{s^2}{4t}}}{t^{\alpha+1}}\, \mathrm{d}t.
\end{equation}
\subsection{Fractional derivatives}

 For $\gamma\geq 0$, the Riemann-Liouville fractional derivative is defined as follows (see Kilbas {\it et al.} (2006)):
\begin{equation}\label{RL}
	\mathcal{D}_t^{\gamma}f(t):=\left\{
	\begin{array}{ll}
		\dfrac{1}{\Gamma{(m-\gamma)}}\displaystyle \frac{\mathrm{d}^{m}}{\mathrm{d}t^{m}}\int^t_{0} \frac{f(s)}{(t-s)^{\gamma+1-m}}\,\mathrm{d}s,\ m-1<\gamma<m,\\\\
		\displaystyle\frac{\mathrm{d}^{m}}{\mathrm{d}t^{m}}f(t),\  \gamma=m,
	\end{array}
	\right.
\end{equation}
where $m$ is a positive integer. Its Laplace transform is given by (see Kilbas {\it et al.} (2006), Eq. (2.2.36))
\begin{equation}\label{laplacerie}
	\mathcal{L}\left(\mathcal{D}_t^{\gamma}f(t);s\right)=s^{\nu}\tilde{f}(s),\  s>0,
\end{equation}
where $\tilde{f}(s)$ denotes the Laplace transform of $f(t)$.

 The Caputo fractional derivative is defined as follows (see Kilbas {\it et al.} (2006)):
\begin{equation}\label{caputo}
	\frac{\mathrm{d}^{\alpha}}{\mathrm{d}t^{\alpha}}f(t)=\left\{
	\begin{array}{ll}
		\dfrac{1}{\Gamma{(1-\alpha)}}\displaystyle\int^t_{0} (t-s)^{-\alpha}f'(s)\,\mathrm{d}s, \ 0<\alpha<1,\\\\
		f'(t),\ \alpha=1,
	\end{array}
	\right.
\end{equation}
whose Laplace transform is given by (see Kilbas {\it et al.} (2006), Eq. (5.3.3))
\begin{equation}\label{laplacecapu}
	\mathcal{L}\left(\frac{\mathrm{d}^{\alpha}}{\mathrm{d}t^{\alpha}}f(t);s\right)=s^{\alpha}\tilde{f}(s)-s^{\alpha-1}f(0).
\end{equation}

\subsection{Stable subordinator and its inverse} A stable subordinator $\{D_\alpha(t)\}_{t\geq0}$, $0<\alpha<1$ is  a non-decreasing L\'evy process. It is characterized by the following Laplace transform: 
\begin{equation}\label{dat}
	\mathbb{E}(e^{-sD_\alpha(t)})=e^{-ts^\alpha},\ s>0.
\end{equation} 
 It is a self-similar process (see Meerschaert and Scheffler (2004)). That is,
 \begin{equation}\label{self}
 	D_{\alpha}(t)\stackrel{d}{=}t^{1/\alpha}D_{\alpha}(1).
 \end{equation}
 Its first passage time, that is,  $Y_\alpha(t)=\inf\{x>0:D_\alpha(x)>t\}$, $t\ge0$ is called the inverse stable subordinator whose Laplace transform is given by (see Meerschaert and Scheffler (2004))
 \begin{equation}\label{yast}
 	\mathbb{E}\big(e^{-sY_\alpha(t)}\big)=E_{\alpha,1}(-st^\alpha).
 \end{equation}

\subsection{GCP and its time-changed variant}
Here, we give some known results for the GCP and its time-changed variant (see Di Crescenzo {\it et al.} (2016), Kataria and Khandakar (2022a)). 

The state probabilities of GCP are given by
\begin{equation}\label{p(n,t)}
	p(n,t)=\sum_{\Omega(k,n)}\prod_{j=1}^{k}\frac{(\lambda_jt)^{x_j}}{x_j!}e^{-\lambda_j t}, \ n\ge0,
\end{equation}
where $\Omega(k,n)=\{(x_1,x_2,\ldots,x_k):\sum_{j=1}^{k}jx_j=n,\ x_j\in \mathbb{N}_0\}$. Here, $\mathbb{N}_0$ denotes the set of non-negative integers. Also, its pgf is given by
\begin{equation}\label{pgfmt}
	G(u,t)=\mathbb{E}\left(u^{M(t)}\right)=\exp\bigg(-\sum_{j=1}^{k}\lambda_{j}(1-u^{j})t\bigg),\ |u|\le 1,
\end{equation}
and moment generating function is given by
\begin{equation}\label{mgfmt}
	\mathbb{E}\left(e^{uM(t)}\right)=\exp\bigg(-\sum_{j=1}^{k}\lambda_{j}(1-e^{uj})t\bigg), \ u\in \mathbb{R}.
\end{equation}
Its mean, variance and covariance are 
\begin{equation}\label{covgcp}
\mathbb{E}(M(t))=c_1t,\ \ \operatorname{Var}(M(t))=c_2t,\ \	\operatorname{Cov}(M(s), M(t))=c_{2}s, \ 0<s\le t,
\end{equation}
respectively, where $c_{1}=\sum_{j=1}^{k}j\lambda_j$ and $c_{2}=\sum_{j=1}^{k}j^{2}\lambda_j$.

Let $B(\beta,\beta+1)$ and $B(\beta,\beta+1;s/t)$ denote the beta function and the incomplete
beta function, respectively. The mean and covariance of a time-changed variant of the GCP, namely, the GFCP are given by
\begin{equation}\label{meanvargfcp}
	\mathbb{E}(M^\beta (t))=\frac{c_1t^\beta}{\Gamma(\beta+1)}
\end{equation}
and 
\begin{align}
	\operatorname{Cov}(M^\beta(s), M^{\beta}(t))&=\Big(\frac{c_1}{\Gamma(\beta+1)}\Big)^2\left(\beta B(\beta,\beta+1) s^{2\beta}+\beta t^{2\beta}B(\beta,\beta+1;s/t)-(ts)^{\beta}\right)\nonumber\\
	&\ \ +\frac{c_2s^\beta}{\Gamma(\beta+1)}, \ 0<\beta<1,\label{covmb}
\end{align}
respectively.

The pmf of a time-changed variant of the GCP, namely, the GSFCP is given by (see Kataria {\it et al.} (2022), Eq. (6.14))
\begin{equation}\label{pntsp}
\mathrm{Pr}\{M(D_\beta(t))=n\}=	\sum_{\Omega(k,n)}\Big(\prod_{j=1}^k\frac{{\lambda_{j}}^{x_j}}{x_j!}\Big)(-{\partial}_\Lambda)^{z_k}e^{{-t\Lambda}^\beta }.
\end{equation}
\section{GCP with random drift}\label{driftsec}
 Beghin and D'Ovidio (2014) introduced a drifted version of the Poisson process, that is, $\{N(t)+bt\}_{t>0}$, $b\geq0$ where $\{N(t)\}_{t>0}$ is a Poisson process with intensity $\lambda>0$. The probability law for this process and its governing differential equation are obtained by them.

First, we consider a drifted version of the GCP $\{M(t)\}_{t>0}$, that is, $\{M(t)+bt\}_{t>0}$.

 Using \eqref{mgfmt}, its Laplace transform is given by
\begin{equation}\label{lapdrft}
	\mathbb{E}(e^{-s (M(t)+ bt)})= e^{-s bt}e^{-t(\sum_{j=1}^{k}\lambda_j(1-e^{-s j}))}, \  s>0.
\end{equation}
\begin{theorem}
	The probability law  of drifted process $\{M(t)+bt\}_{t>0}$ is given by  
	\begin{equation}\label{10.2}
		p(y,t)= \sum_{n=0}^{\infty}\sum_{\Omega(k,n)}\Big(\prod_{j=1}^{k}\frac{(\lambda_j t)^{x_j}}{x_j!}\Big)e^{-\Lambda t}\delta(y-n-bt),\  y\geq bt,
	\end{equation} 
where $\delta(\cdot)$ is the Dirac delta function. It solves the following differential equation:
	\begin{equation}\label{10.3}
		\bigg(\frac{\partial}{\partial t}+b\frac{\partial}{\partial y}\bigg)p(y,t)=-\Lambda p(y,t)+\sum_{j=1}^{n\wedge k}\lambda_j p(y-j,t)
	\end{equation}
	with initial and boundary condition
	$
	p(y,0)=\delta(y)$ and $
	p(0,t)=\delta(bt)e^{-\Lambda t}$, respectively. 
\end{theorem}
\begin{proof}
 Let us obtain the Laplace transform of \eqref{10.2} as follows:
	\begin{align}\label{pytl}
		\tilde{p}(s,t)&= \int_{0}^{\infty}e^{-s y}p(y,t)\, \mathrm{d}y\nonumber\\
		&=e^{-s bt}\sum_{n=0}^{\infty}e^{-s n}\sum_{\Omega(k,n)}\Big(\prod_{j=1}^{k}\frac{(\lambda_j t)^{x_j}}{x_j!}\Big)e^{-\Lambda t}\nonumber\\
		&= e^{-s bt}e^{-t(\sum_{j=1}^{k}\lambda_j(1-e^{-s j}))},
	\end{align}
where in the last step we have used \eqref{mgfmt}.	From \eqref{lapdrft}, \eqref{pytl} and by using the  uniqueness of Laplace transform, we establish \eqref{10.2}. 

By taking the Laplace transform with respect to $y$ on both sides of \eqref{10.3}, we get
	\begin{align*}
		\frac{\partial }{\partial t}\tilde{p}(s, t)+bs \tilde{p}(s,t)-\delta(bt)e^{-\Lambda t}&=\int_{0}^{\infty}e^{-s y}\bigg(-\Lambda p(y,t)+\sum_{j=1}^{n\wedge k}\lambda_j p(y-j,t)\bigg)\, \mathrm{d}y\\&= -\Lambda \tilde{p}(s, t) + \sum_{j=1}^{n \wedge k} e^{-s j} \lambda_j\tilde{p}(s, t)
	\end{align*}
	which reduces to 
	\begin{equation}\label{deltat}
		\frac{\partial }{\partial t}\tilde{p}(s, t)=\Big(-bs-\Lambda+\sum_{j=1}^{n \wedge k} e^{-s j} \lambda_j\Big)\tilde{p}(s, t)
	\end{equation}
with initial condition $\tilde{p}(s,0)=1$.  It can be shown that the solution of \eqref{deltat} coincides with \eqref{lapdrft}. 
	This completes the proof.
\end{proof}
 Beghin and D'Ovidio (2014) considered a generalization of the drifted process $\{N(t)+bt\}_{t>0}$. That is, they studied the following drifted process: 
\begin{equation}\label{bprorndf}
	N(D_\gamma(Y_\beta(t)))+bD_\alpha(Y_\beta(t)), \ b\ge0, \ \alpha\in (0,1], \ \gamma\in (0,1], \ \beta\in (0,1], \ t>0,
\end{equation}
 where  $\{N(t)\}_{t>0}$ is a Poisson process,  $\{D_\alpha(t)\}_{t>0}$, $0<\alpha<1$, is a stable subordinator and $\{Y_\beta(t)\}_{t>0}$, $0<\beta<1$ is an inverse stable subordinator, and all these processes are mutually independent.

Next, we consider a similar drifted version of the GCP $\{M(t)\}_{t>0}$ in the following form:
\begin{equation}\label{prorndf}
	M^\alpha_{\gamma,\beta}(t)\coloneqq M(D_\gamma(Y_\beta(t)))+bD_\alpha(Y_\beta(t)),\ t>0,
\end{equation} 
where all the component processes are independent from each other.
\begin{remark}
	For $b=0$, the process defined in $\eqref{prorndf}$ reduces to the GSTFCP (see Kataria {\it et al.} (2022), Section 6). For $k=1$,  the process defined in $\eqref{prorndf}$ reduces to the process given in $\eqref{bprorndf}$.
\end{remark}  
Let $g_\alpha(\cdot,t)$ and $h_\beta(\cdot,t)$ be the density function of $\{D_\alpha(t)\}_{t>0}$ and $\{Y_\beta(t)\}_{t>0}$, respectively. The translation operator for an analytic function $f:\mathbb{R}\to \mathbb{R}$ is defined as
\begin{equation}\label{translationo}
	e^{y\partial_x}f(x)=f(x+y),\ y\in \mathbb{R}.
\end{equation} 
\begin{theorem}
	The probability law of drifted process $\{M^\alpha_{\gamma,\beta}(t)\}_{t>0}$ is given by
{\small	\begin{equation}\label{drifted}
		\mathrm{Pr}\{M^\alpha_{\gamma,\beta}(t)
		\in \mathrm{d}x\}=\sum_{n=0}^{\infty}\sum_{\Omega(k,n)}\Big(\prod_{j=1}^k\frac{{\lambda_{j}}^{x_j}}{x_j!}\Big)(-{\partial}_\Lambda)^{z_k}\int_{0}^{\infty}e^{{-s\Lambda}^\gamma }g_\alpha(x-n,b^\alpha s)h_{\beta}(s,t)\, \mathrm{d}s\, \mathrm{d}x
	\end{equation}}
	which solves
	\begin{equation}
		\bigg(\dfrac{\mathrm{d}^{\beta}}{\mathrm{d}t^{\beta}}+b^\alpha D_x^{\alpha}+\Big(\sum_{j=1}^{k}\lambda_j(I-T^j)\Big)^\gamma\bigg)w(x,t)=0, \  x\geq0,\label{eqndrift}
	\end{equation}
	with $w(x,0)=\delta(x)$ and
	\begin{equation*}
		(I-T)^\gamma=\sum_{r=0}^{\infty}(-1)^r\begin{pmatrix}
			\gamma \\ r
		\end{pmatrix}T^r,
	\end{equation*}
	where\begin{equation*}
		T^r=\begin{cases}
			e^{-r\partial_x}, \,b>0,\\
			B^r, \, b=0.
		\end{cases}
	\end{equation*}
	 Here, $B$ is the backward shift operator, $D_x^{\alpha}$ is the Riemann-Liouville fractional derivative defined in \eqref{RL} and $\dfrac{\mathrm{d}^{\beta}}{\mathrm{d}t^{\beta}}$ is the Caputo fractional derivative given in \eqref{caputo}.
\end{theorem}
\begin{proof}
	Observe that
	\begin{equation}\label{lapmagb}
		\mathbb{E}\big(e^{-\eta M^\alpha_{\gamma,\beta}(t)}\big)=\mathbb{E}\big(\mathbb{E}(e^{-\eta M(D_\gamma(Y_\beta(t)))-b\eta D_\alpha(Y_\beta(t))}|Y_\beta(t))\big).
	\end{equation}
	By using \eqref{dat}, we have
	\begin{equation}\label{msn}
		\mathbb{E}\big(e^{-b\eta D_\alpha(Y_{\beta}(t))}\big|Y_\beta(t)\big)=\exp(-(b\eta)^\alpha Y_\beta (t)).
	\end{equation}
Also,
	\begin{align}\label{gsn}
		\mathbb{E}\big(e^{-\eta M(D_\gamma(Y_\beta(t)))}|Y_\beta(t)\big)&=\mathbb{E}\big(\mathbb{E}(e^{-\eta M(D_\gamma(Y_\beta(t)))}|D_\gamma(Y_\beta(t)))|Y_\beta(t)\big)\nonumber\\
		&=\mathbb{E}\Big(e^{-\sum_{j=1}^{k}\lambda_j(1-e^{-\eta j})D_\gamma(Y_\beta(t))}\big|Y_\beta(t)\Big),\,  \text{(using \eqref{mgfmt})}\nonumber\\
		&=\exp\Big(-Y_\beta(t)\Big(\sum_{j=1}^{k}\lambda_{j}(1-e^{-\eta j})\Big)^\gamma \Big),
	\end{align}
	where in the last step we have used \eqref{dat}.
	On using \eqref{msn} and \eqref{gsn} in  \eqref{lapmagb}, we get
	\begin{align}
		\mathbb{E}(e^{-\eta M^\alpha_{\gamma,\beta}(t)})&=\mathbb{E}\Big(e^{-\big(b^\alpha\eta^\alpha +(\sum_{j=1}^{k}\lambda_{j}(1-e^{-\eta j}))^\gamma\big) Y_\beta(t)}\Big)\nonumber\\
	  	&=E_{\beta,1}\Big(-\big(b^\alpha\eta^\alpha +\Big(\sum_{j=1}^{k}\lambda_{j}(1-e^{-\eta j})\Big)^\gamma\big)t^{\beta}\Big),\label{laplace1}
	\end{align}
	where the last step follows from \eqref{yast}.
	
	For any Borel set $B\in\mathbb{B}_{\mathbb{R_+}}$, we have
	\begin{align*}
		\mathrm{Pr}&\{M^\alpha_{\gamma,\beta}(t)\in B\}\\
		&=\sum_{n=0}^{\infty}\mathrm{Pr}\left\{D_\alpha(Y_\beta(t))\in\frac{B-n}{b}\right\}\mathrm{Pr}\{M(D_\gamma(Y_\beta(t)))=n\}\\
		&=\int_{0}^{\infty}\sum_{n=0}^{\infty}\mathrm{Pr}\left\{D_\alpha(s)\in\frac{B-n}{b}\right\}\mathrm{Pr}\{M(D_\gamma(s))=n\}h_{\beta}(s,t)\, \mathrm{d}s\nonumber\\
		&=\int_{0}^{\infty}\sum_{n=0}^{\infty}\sum_{\Omega(k,n)}\Big(\prod_{j=1}^k\frac{{\lambda_{j}}^{x_j}}{x_j!}\Big)(-{\partial}_\Lambda)^{z_k}e^{{-s\Lambda}^\gamma }\mathrm{Pr}\left\{D_\alpha(s)\in\frac{B-n}{b}\right\}h_{\beta}(s,t)\, \mathrm{d}s,\ (\text{using} \ \eqref{pntsp})\nonumber\\
		&=\int_{0}^{\infty}\sum_{n=0}^{\infty}\sum_{\Omega(k,n)}\Big(\prod_{j=1}^k\frac{{\lambda_{j}}^{x_j}}{x_j!}\Big)(-{\partial}_\Lambda)^{z_k}e^{{-s\Lambda}^\gamma }\mathrm{Pr}\left\{D_\alpha(b^\alpha s)\in B-n\right\}h_{\beta}(s,t)\, \mathrm{d}s,\ (\text{using} \ \eqref{self})\nonumber\\
		&=\int_{B}\sum_{n=0}^{\infty}\sum_{\Omega(k,n)}\Big(\prod_{j=1}^k\frac{{\lambda_{j}}^{x_j}}{x_j!}\Big)(-{\partial}_\Lambda)^{z_k}\int_{0}^{\infty}e^{{-s\Lambda}^\gamma}g_\alpha(x-n,b^\alpha s)h_{\beta}(s,t)\, \mathrm{d}s\, \mathrm{d}x.
	\end{align*}
Therefore,
{\small	\begin{equation}\label{lapm}
		\mathrm{Pr}\{M^\alpha_{\gamma,\beta}(t)\in \mathrm{d}x\}=\sum_{n=0}^{\infty}\sum_{\Omega(k,n)}\Big(\prod_{j=1}^k\frac{{\lambda_{j}}^{x_j}}{x_j!}\Big)(-{\partial}_\Lambda)^{z_k}\int_{0}^{\infty}e^{{-s\Lambda}^\gamma}g_\alpha(x-n,b^\alpha s)h_{\beta}(s,t)\, \mathrm{d}s\, \mathrm{d}x.
	\end{equation}}
On taking the Laplace transform in \eqref{lapm}, we get
\begin{align*}
	\int_{0}^{\infty}e^{-\eta x}&\mathrm{Pr}\{M^\alpha_{\gamma,\beta}(t)\in \mathrm{d}x\}\\
	&=\sum_{n=0}^{\infty}\sum_{\Omega(k,n)}\Big(\prod_{j=1}^k\frac{{\lambda_{j}}^{x_j}}{x_j!}\Big)(-{\partial}_\Lambda)^{z_k}\int_{0}^{\infty}\int_{0}^{\infty}e^{-\eta x}e^{{-s\Lambda}^\gamma}g_\alpha(x-n,b^\alpha s)h_{\beta}(s,t)\, \mathrm{d}s\, \mathrm{d}x\\
	&=\sum_{n=0}^{\infty}\sum_{\Omega(k,n)}\Big(\prod_{j=1}^k\frac{{\lambda_{j}}^{x_j}}{x_j!}\Big)(-{\partial}_\Lambda)^{z_k}e^{-\eta n}\int_{0}^{\infty}e^{-s(\Lambda^\gamma+b^\alpha \eta^\alpha)}h_\beta(s,t)\, \mathrm{d}s\\
	&=\sum_{n=0}^{\infty}\sum_{\Omega(k,n)}\Big(\prod_{j=1}^k\frac{{\lambda_{j}}^{x_j}}{x_j!}\Big)(-{\partial}_\Lambda)^{z_k}e^{-\eta n}E_{\beta,1}(-t^\beta(\Lambda^\gamma+b^\alpha \eta^\alpha)),\  (\text{using} \ \eqref{yast})\\
	&=\sum_{n=0}^{\infty}\sum_{\Omega(k,n)}\Big(\prod_{j=1}^{k}\frac{(-\partial_{\Lambda} \lambda_je^{-\eta j})^{x_j}}{x_j!}\Big)E_{\beta,1}(-t^\beta(\Lambda^\gamma+b^\alpha\eta^\alpha))\\
	&=\exp\Big(-\partial_{\Lambda}\sum_{j=1}^{k}\lambda_{j}e^{-\eta j}\Big)E_{\beta,1}(-t^\beta(\Lambda^\gamma+b^\alpha\eta^\alpha))\\
	&=E_{\beta,1}\Big(-\Big(b^\alpha\eta^\alpha +\Big(\sum_{j=1}^{k}\lambda_{j}(1-e^{-\eta j})\Big)^\gamma\Big)t^{\beta}\Big)
\end{align*}
 which coincides with \eqref{laplace1}. Here, the last step follows from \eqref{translationo}. Thus, by uniqueness of the Laplace transform, we conclude that \eqref{drifted} is the probability law of $\{M^\alpha_{\gamma,\beta}(t)\}_{t>0}$.
	
 Note that
{\small\begin{align}
	\int_{0}^{\infty}e^{-\eta x}\Big(&\sum_{j=1}^{k}\lambda_j(I-T^j)\Big)^\gamma w(x,t)\mathrm{d}x\nonumber\\
	&=\int_{0}^{\infty}e^{-\eta x}\sum_{r_1+r_2+\dots+r_k=\gamma}\frac{\gamma!}{r_1! r_2!\dots r_k!}\Big(\prod_{j=1}^{k}(\lambda_j(I-T^j))^{r_j}\Big)w(x,t)\mathrm{d}x\nonumber\\
	&=\gamma!\sum_{r_1+r_2+\dots+r_k=\gamma}\Big(\prod_{j=1}^{k}\frac{\lambda_j^{r_j}}{r_j!}\Big)\int_{0}^{\infty}e^{-\eta x}\Big(\prod_{j=1}^{k}\sum_{l_j=0}^{r_j}\binom{r_j}{l_j}(-1)^{l_j}T^{jl_j}\Big)w(x,t)\mathrm{d}x\nonumber\\
	&=\gamma!\sum_{r_1+r_2+\dots+r_k=\gamma}\Big(\prod_{j=1}^{k}\frac{\lambda_j^{r_j}}{r_j!}\sum_{l_j=0}^{r_j}\binom{r_j}{l_j}(-1)^{l_j}\Big)\int_{0}^{\infty}e^{-\eta x}T^{l_1+2l_2+\dots+kl_k}w(x,t)\mathrm{d}x\nonumber\\
	&=\gamma!\sum_{r_1+r_2+\dots+r_k=\gamma}\Big(\prod_{j=1}^{k}\frac{\lambda_j^{r_j}}{r_j!}\sum_{l_j=0}^{r_j}\binom{r_j}{l_j}(-1)^{l_j}e^{-\eta jl_j}\Big)\int_{0}^{\infty}e^{-\eta x}w(x,t)\mathrm{d}x\nonumber\\
	&=\sum_{r_1+r_2+\dots+r_k=\gamma}\frac{\gamma!}{r_1! r_2!\dots r_k!}\Big(\prod_{j=1}^{k}(\lambda_j(1-e^{-\eta j}))^{r_j}\Big)\int_{0}^{\infty}e^{-\eta x}w(x,t)\mathrm{d}x\nonumber\\
	&=\Big(\sum_{j=1}^{k}\lambda_j(1-e^{-\eta j})\Big)^\gamma\int_{0}^{\infty}e^{-\eta x}w(x,t)\mathrm{d}x.\label{xxxx}
\end{align}}
Let
\begin{equation*}
\widetilde{w}(\eta,t)=\int_{0}^{\infty}e^{-\eta x}w(x,t)\, \mathrm{d}x\ \text{and}\ 	\widetilde{\widetilde {w}}(\eta, \mu)=\int_{0}^{\infty}e^{-\mu t}\int_{0}^{\infty}e^{-\eta x}w(x,t)\, \mathrm{d}x\, \mathrm{d}t.
\end{equation*}
Then, by taking the double Laplace transform in \eqref{eqndrift}, we get
\begin{equation*}
	\mu^\beta \widetilde{\widetilde {w}}(\eta,\mu)-\mu^{\beta-1}\widetilde{w}(\eta,0)+b^\alpha\eta^\alpha\widetilde{\widetilde {w}}(\eta,\mu)=-\bigg(\sum_{j=1}^{k}\lambda_{j}(I-e^{-\eta j})\bigg)^\gamma\widetilde{\widetilde {w}}(\eta,\mu),
\end{equation*}
 where we have used (\ref{laplacerie}), \eqref{laplacecapu}  and \eqref{xxxx}. Equivalently, \begin{equation}\label{formula1}
	\widetilde{\widetilde {w}}(\eta,\mu)=\mu^{\beta-1}\Big(\mu^\beta+b^\alpha\eta^\alpha+\Big(\sum_{j=1}^{k}\lambda_j(1-e^{-\eta j})\Big)^\gamma\Big)^{-1}\widetilde{w}(\eta,0).
\end{equation}
By taking the Laplace transform on both sides of (\ref{laplace1}) and using \eqref{mi}, we get
\begin{align}
	\int_{0}^{\infty}e^{-\mu t}\mathbb{E}(e^{-\eta M^\alpha_{\gamma,\beta}(t)})\, \mathrm{d}t&=\int_{0}^{\infty}e^{-\mu t}E_{\beta,1}\Big(-\big(b^\alpha\eta^\alpha +\Big(\sum_{j=1}^{k}\lambda_{j}(1-e^{-\eta j})\Big)^\gamma\big)t^{\beta}\Big)\, \mathrm{d}t\nonumber\\
	&=\mu^{\beta-1}\Big(\mu^\beta+b^\alpha\eta^\alpha+\Big(\sum_{j=1}^{k}\lambda_j(1-e^{-\eta j})\Big)^\gamma\Big)^{-1}.\label{formula2}
\end{align}
Finally, for $w(x,0)=\delta(x)$,  (\ref{formula1}) and (\ref{formula2}) coincide, which leads to the required result. 
\end{proof}
 For $\beta=1$, the process defined in \eqref{prorndf} reduces to $M^\alpha_{\gamma}(t)=M(D_\gamma(t))+bD_\alpha(t)$, $\alpha\in(0,1]$, $\gamma\in(0,1]$. This is because $Y_\beta(t)=t$ for $\beta=1$.	The hitting time of $\{M^\alpha_{\gamma}(t)\}_{t\geq0}$ is defined as
	\begin{equation*}
		H(t)=\inf\{s\geq0:M^\alpha_{\gamma}(s)>t\}, \ t\ge0.
	\end{equation*}
	
	In the next result, we obtain a governing differential equation of the density function of $\{H(t)\}_{t\geq0}$.
	\begin{theorem}
		The density function  $g(x,t)=\mathrm{Pr}\{H(t)\in\mathrm{d}x\}/\mathrm{d}x$ solves the following differential equation:
		\begin{equation}\label{hitingden}
			b^\alpha\frac{\mathrm{d}^\alpha}{\mathrm{d}t^\alpha} w(x,t)+\Big(\sum_{j=1}^{k}\lambda_j(I-T^j)\Big)^\gamma w(x,t)=-\frac{\partial }{\partial x}w(x,t), \, x>0, \, t>0,
		\end{equation}
		with initial condition $w(x,0)=\delta(x)$ and boundary condition
		{\small\begin{equation*}
			w(0,t)=\frac{-\gamma\Lambda^\gamma}{\Gamma(1-\gamma)}\sum_{n=0}^{\infty}\frac{\Gamma(n-\gamma)}{\Lambda^{n}}\sum_{n_1+n_2+\dots+n_k=n}\Big(\prod_{j=1}^{k}\frac{\lambda_j^{n_j}}{n_j!}\Big)\mathcal{H}(t-(n_1+2n_2+\dots+kn_k)).
		\end{equation*}}
		Here, $\mathcal{H}(\cdot)$ is the Heaviside step function and \begin{equation*}
			T=\begin{cases}
				e^{-\partial_t}, \, \ b>0,\\
				B, \, \ b=0,
			\end{cases}
		\end{equation*}
		is the shift operator.
	\end{theorem}
	\begin{proof}
		Note that 
		\begin{align}
			\int_{0}^{\infty}e^{-\eta t}\mathrm{Pr}\{H(t)>x\}\mathrm{d}t&=\int_{0}^{\infty}e^{-\eta t}\mathrm{Pr}\{M^\alpha_{\gamma}(x)<t\}\, \mathrm{d}t \nonumber\\
			&=\eta^{-1}\int_{0}^{\infty}e^{-\eta t}\frac{\partial}{\partial t}\mathrm{Pr}\{M^\alpha_{\gamma}(x)<t\}\, \mathrm{d}t \nonumber\\
			&=\eta^{-1}\exp\Big(-\Big(b^\alpha\eta^\alpha  +\Big(\sum_{j=1}^{k}\lambda_j(1-e^{-\eta j})\Big)^\gamma\Big) x\Big),\label{hitting1}
		\end{align}
		where the last step follows on using \eqref{laplace1} with $\beta=1$.
		So, we have
		\begin{align}
			\widetilde{g}(x,\eta)&=\int_{0}^{\infty}e^{-\eta t}g(x,t) \mathrm{d}t\nonumber\\
			&=\int_{0}^\infty e^{-\eta t}\bigg(\frac{\partial}{\partial x}\mathrm{Pr}\{H(t)<x\}\bigg)\, \mathrm{d}t\nonumber\\
			&=-\frac{\partial}{\partial x}\Big(\eta^{-1}\exp\Big(-\Big(b^\alpha\eta^\alpha  +\Big(\sum_{j=1}^{k}\lambda_j(1-e^{-\eta j})\Big)^\gamma\Big) x\Big)\Big),\ (\text{using} \ \eqref{hitting1})\nonumber\\
			&=\eta^{-1}\bigg( b^\alpha\eta^\alpha+\bigg(\sum_{j=1}^{k}\lambda_j(1-e^{-\eta j})\bigg)^\gamma\bigg)  \exp\Big(-\Big(b^\alpha\eta^\alpha  +\Big(\sum_{j=1}^{k}\lambda_j(1-e^{-\eta j})\Big)^\gamma\Big) x\Big).\label{lapdrift}
		\end{align}

		Now, by taking the Laplace transform on both sides of \eqref{lapdrift}, we get
		\begin{equation}\label{dg}
			\widetilde{\widetilde{g}}(\mu,\eta)=\frac{b^\alpha \eta^{\alpha-1}+\eta^{-1}\big(\sum_{j=1}^{k}\lambda_j(1-e^{-\eta j})\big)^\gamma}{\mu +b^\alpha\eta^\alpha+\big(\sum_{j=1}^{k}\lambda_j(1-e^{-\eta j})\big)^\gamma}.
		\end{equation} 
	Further, by taking the Laplace transform on both sides of (\ref{hitingden}), we get
		\begin{equation*}
			b^\alpha\eta^\alpha \widetilde{w}(x,\eta)-b^\alpha\eta^{\alpha-1}\delta(x)+\Big(\sum_{j=1}^{k}\lambda_j(1-e^{-\eta j})\Big)^\gamma\widetilde{w}(x,\eta)
			=-\frac{\partial}{\partial x}\widetilde{w}(x,\eta),
		\end{equation*}
where we have used \eqref{laplacecapu} and \eqref{xxxx}. Also, we have
		\begin{align*}
			\int_{0}^{\infty}e^{-\mu x}\frac{\partial}{\partial x}\widetilde{w}(x,\eta)\, \mathrm{d}x&=\int_{0}^{\infty}\frac{\partial}{\partial x}(e^{-\mu x}\widetilde{w}(x,\eta))\, \mathrm{d}x-\int_{0}^{\infty}\widetilde{w}(x,\eta)\frac{\partial}{\partial x}e^{-\mu x}\, \mathrm{d}x\\
			&=-\widetilde{w}(0,\eta)+\mu\widetilde{\widetilde{w}}(\mu,\eta).
		\end{align*}
Here,
		\begin{align}
			\widetilde{w}(0,\eta)&=\int_{0}^{\infty}e^{-\eta t}w(0,t)\mathrm{d}t\nonumber\\
			&=\int_{0}^{\infty}e^{-\eta t}\bigg(\frac{-\gamma\Lambda^\gamma}{\Gamma(1-\gamma)}\sum_{n=0}^{\infty}\frac{\Gamma(n-\gamma)}{\Lambda^{n}}\sum_{n_1+n_2+\dots+n_k=n}\Big(\prod_{j=1}^{k}\frac{\lambda_j^{n_j}}{n_j!}\Big)\mathcal{H}\Big(t-\sum_{j=1}^{k}jn_j\Big)\bigg) \mathrm{d}t\nonumber\\
			&=\frac{\Lambda^\gamma}{\eta}-\frac{\gamma\Lambda^\gamma}{\eta\Gamma(1-\gamma)}\sum_{n=1}^{\infty}\frac{\Gamma(n-\gamma)}{\Lambda^{n}}\sum_{n_1+n_2+\dots+n_k=n}\prod_{j=1}^{k}\frac{(\lambda_je^{-\eta j})^{n_j}}{n_j!}\nonumber\\
			&=\frac{\Lambda^\gamma}{\eta}-\frac{\gamma\Lambda^\gamma}{\eta\Gamma(1-\gamma)}\sum_{n=1}^{\infty}\frac{\Gamma(n-\gamma)}{n!}\Big(\frac{1}{\Lambda}\sum_{j=1}^{k}\lambda_je^{-\eta j}\Big)^n\nonumber\\
			&=\int_{0}^{\infty}\frac{\gamma}{\Gamma(1-\gamma)}\bigg(\frac{1-e^{-s\Lambda}}{\eta s^{\gamma+1}}-\frac{e^{-s\Lambda}}{\eta}\sum_{n=1}^{\infty}\frac{s^{n-\gamma-1}}{n!}\Big(\sum_{j=1}^{k}\lambda_je^{-\eta j}\Big)^n\bigg)\mathrm{d}s\label{mmmm}\\
			&=\frac{\gamma}{\eta \Gamma(1-\gamma)}\int_{0}^{\infty}\frac{1 }{s^{\gamma+1}}\bigg(1-e^{-s\Lambda}\sum_{n=0}^{\infty}\frac{s^n}{n!}\Big(\sum_{j=1}^{k}\lambda_je^{-\eta j}\Big)^n\bigg)\mathrm{d}s\nonumber\\
			&=\frac{\gamma}{\eta \Gamma(1-\gamma)}\int_{0}^{\infty}\frac{1 }{s^{\gamma+1}}\bigg(1-\exp\Big(-s\Big(\sum_{j=1}^{k}\lambda_j(1-e^{-\eta j})\Big)\Big)\bigg)\mathrm{d}s\nonumber\\
			&=\eta^{-1}\bigg(\sum_{j=1}^{k}\lambda_j(1-e^{-\eta j})\bigg)^\gamma.\nonumber
		\end{align}
To obtain \eqref{mmmm} and the last step, we have used the following result (see Applebaum (2009), p. 53):
	\begin{equation*}
		u^\alpha=\frac{\alpha}{\Gamma(1-\alpha)}\int_{0}^{\infty}(1-e^{-su})\frac{\mathrm{d}s}{s^{1+\alpha}},\ 0<\alpha<1,\, u\ge0.
	\end{equation*}
		Therefore, by taking the double Laplace transform on both sides of (\ref{hitingden}), we get
		\begin{equation*}
			b^\alpha\eta^\alpha\widetilde{\widetilde {w}}(\mu,\eta)-b^\alpha\eta^{\alpha-1}+\Big(\sum_{j=1}^{k}\lambda_j(1-e^{-\eta j})\Big)^\gamma\widetilde{\widetilde {w}}(\mu,\eta)=-\mu\widetilde{\widetilde {w}}(\mu,\eta)+\eta^{-1}\Big(\sum_{j=1}^{k}\lambda_j(1-e^{-\eta j})\Big)^\gamma.
		\end{equation*}
Thus,
		\begin{equation}\label{dw}
			\widetilde{\widetilde {w}}(\mu,\eta)=\frac{b^\alpha\eta^{\alpha-1}+\eta^{-1}\big(\sum_{j=1}^{k}\lambda_{j}(1-e^{-\eta j})\big)^\gamma}{\mu+b^\alpha\eta^\alpha+\big(\sum_{j=1}^{k}\lambda_j(1-e^{-\eta j})\big)^\gamma}.
		\end{equation}
Finally, from \eqref{dg} and \eqref{dw}, we have $\widetilde{\widetilde {w}}=\widetilde{\widetilde {g}}$. This completes the proof.
	\end{proof}
\section{GCP with different Brownian clocks}\label{sec2}
Beghin and Orsingher (2012) studied different versions of the time-changed Poisson process where the time changing components are the elastic Brownian motion, the Brownian sojourn time on
 positive half-line, the first-passage time of a Brownian
motion with or without drift and the Bessel process. Several distributional properties such as the pmf, mean, variance, pgf, {\it etc.} are derived for these time-changed processes.

 Here, we do a similar study for the GCP.

\subsection{GCP at Brownian first passage-time}\label{secfirstpas}
The first-passage time $\{Z(t)\}_{t>0}$ of a standard Brownian motion $\{B(t)\}_{t>0}$ is defined by $Z(t)=\inf\{s>0: B(s)=t\}$. The pdf $f(\cdot,t)$ of $\{Z(t)\}_{t>0}$ is given by 
\begin{equation}\label{qts}
	f(s,t)=\frac{te^{-t^{2}/2s}}{\sqrt{2\pi s^{3}}}, \  s>0,\ t>0.
\end{equation}
The pdf $f(s,t)$ solves (see D'Ovidio and Orsingher (2011), Eq. (4.2)):
\begin{equation}\label{qt}
	\frac{\partial^{2}}{\partial t^{2}}f(s,t)=2\frac{\partial}{\partial s}f(s,t).
\end{equation}
Its Laplace transform is given by 
\begin{equation}\label{lqts}
	\int_{0}^{\infty}e^{-\gamma s}f(s,t)\,\mathrm{d}s=e^{-t\sqrt{2\gamma}}, \ \gamma>0.
\end{equation}

Let us consider the following time-changed process:
\begin{equation}\label{processde}
	\hat{M}(t)=M(Z(t)),\, t>0,
\end{equation}
where the GCP $\{M(t)\}_{t>0}$ is independent of $\{Z(t)\}_{t>0}$. 

For $k=1$, the process $\{\hat{M}(t)\}_{t>0}$ reduces to the Poisson process time-changed by $\{Z(t)\}_{t>0}$ (see Beghin and Orsingher (2012), Section 3).

\begin{theorem}
	The	state probabilities  $\hat{p}(n,t)=\mathrm{Pr}\{\hat{M}(t)=n\}$ satisfy the following system of differential equations:
	\begin{equation*}
		\frac{\mathrm{d}^{2}}{\mathrm{d}t^{2}}\hat{p}(n,t)=2\Big(\Lambda\hat{p}(n,t)-\sum_{j=1}^{n\wedge k}\lambda_{j}\hat{p}(n-j,t)\Big),\, n\ge0.
	\end{equation*}
\end{theorem}
\begin{proof}
	 From (\ref{qts}) and \eqref{processde}, we have
	\begin{equation}\label{dis}
		\hat{p}(n,t)=\int_{0}^{\infty}p(n,s)f(s,t)\,\mathrm{d}s.
		\end{equation}
		 On taking the second order derivative with respect to $t$ in \eqref{dis}, we obtain
		 \begin{align*}
		\frac{\mathrm{d}^{2}}{\mathrm{d}t^{2}}\hat{p}(n,t)&=\int_{0}^{\infty}p(n,s)\frac{\partial^{2}}{\partial t^{2}}f(s,t)\, \mathrm{d}s\nonumber\\
		&=2\int_{0}^{\infty}p(n,s)\frac{\partial
		}{\partial s}f(s,t)\,\mathrm{d}s, \, (\text{using}\,(\ref{qt}))\nonumber\\
		&=2p(n,s)f(s,t)\big|_{s=0}^{\infty}-2\int_{0}^{\infty}f(s,t)\frac{\mathrm{d}}{\mathrm{d}s}p(n,s)\,\mathrm{d}s\nonumber\\
		&=-2\int_{0}^{\infty}f(s,t)\bigg(-\Lambda p(n,s)+\sum_{j=1}^{n\wedge k}\lambda_{j}p(n-j,s)\bigg)\,\mathrm{d}s,
	\end{align*}
where in the penultimate step we have used $\lim\limits_{s\to 0}f(s,t)=0$ and \eqref{gcpdE}. Finally, the result follows on using \eqref{dis}.
\end{proof}
Next, we obtain the state probabilities of $\{\hat{M}(t)\}_{t>0}$
in terms of the modified Bessel function. 
\begin{proposition}
	The	state probabilities of $\{\hat{M}(t)\}_{t>0}$ are given by
\begin{equation*}
		\hat{p}(n,t)=\sum_{\Omega(k,n)}\Big(\prod_{j=1}^{k}\frac{\lambda_{j}^{x_{j}}}{x_{j}!}\Big)\frac{t^{z_{k}+\frac{1}{2}}\Lambda^{\frac{1}{4}-\frac{z_k}{2}}2^{\frac{3}{4}-\frac{z_{k}}{2}}}{\sqrt{\pi}}K_{z_{k}-\frac{1}{2}}\left(t\sqrt{2\Lambda}\right),\, n\ge0,
\end{equation*}
 where $z_k=x_1+ x_{2}+\cdots+x_{k}$ and  $K_\alpha(\cdot)$ is the modified Bessel function of index $\alpha$ defined in \eqref{modifiedbes}. 
\end{proposition}
\begin{proof}
On using \eqref{p(n,t)} and \eqref{qts} in (\ref{dis}), we get
\begin{align}
	\hat{p}(n,t)
	&=\int_{0}^{\infty}\sum_{\Omega(k,n)}\Big(\prod_{j=1}^{k}\frac{\lambda_{j}^{x_{j}}}{x_{j}!}\Big)\frac{t}{\sqrt{2\pi}}s^{z_{k}-3/2}e^{-\Lambda s}e^{-t^{2}/2s}\, \mathrm{d}s\nonumber\\
	&=\sum_{\Omega(k,n)}\Big(\prod_{j=1}^{k}\frac{\lambda_{j}^{x_{j}}}{x_{j}!}\Big)\frac{t}{\sqrt{2\pi}}\int_{0}^{\infty}s^{z_{k}-3/2}e^{-\Lambda s}e^{-t^{2}/2s}\, \mathrm{d}s.\label{pcapnt}
\end{align}
By using the following result (see  Gradshteyn and Ryzhik (2007), p. 368):
\begin{equation}\label{rule}
	\int_{0}^{\infty}x^{\alpha-1}e^{-\beta/x-\gamma x}\mathrm{d}x=2(\beta/\gamma)^{\alpha/2}K_{\alpha}(2\sqrt{\beta \gamma}),\ \beta>0,\ \gamma>0
\end{equation}
in \eqref{pcapnt}, we get the required pmf.
\end{proof}

Let $\hat{G}(u,t)=\mathbb{E}(u^{\hat{M}(t)})$.
From \eqref{processde}, we have 
\begin{align*}
\hat{G}(u,t)&=\int_{0}^{\infty}G(u,s)f(s,t)\,\mathrm{d}s\\
&=\int_{0}^{\infty}e^{\left(-\sum_{j=1}^{k}\lambda_j(1-u^j)s\right)}\, f(s,t)\, \mathrm{d}s, \,\,  (\text{using} \,\eqref{pgfmt}\\
&=e^{-t\sqrt{2\sum_{j=1}^{k}\lambda_{j}(1-u^j)}}, \ |u|\le 1,
\end{align*}
where in the last step we have used \eqref{lqts}.

 Further, we have
\begin{align*}
\mathbb{E}(\hat{M}(t))&=\frac{\partial }{\partial u}e^{-t\sqrt{2\sum_{j=1}^{k}\lambda_{j}(1-u^j)}}\bigg|_{u=1}\\
&=e^{-t\sqrt{2\sum_{j=1}^{k}\lambda_{j}(1-u^j)}}\frac{t\sum_{j=1}^{k}j\lambda_{j}u^{j-1}}{\sqrt{2\sum_{j=1}^{k}\lambda_j(1-u^j)}}\Bigg|_{u=1}=\infty.
\end{align*}
 Thus, the first moment of $\{\hat{M}(t)\}_{t>0}$ is infinite.

So, we consider the first passage time $Z^{\mu}(t)=\inf\{s>0: B^{\mu}(s)=t\}$ of a Brownian motion $\{B^{\mu}(t)\}_{t>0}$ with drift $\mu\in \mathbb{R}$ as a time-changed component in GCP. The pdf of $\{Z^{\mu}(t)\}_{t>0}$ is given by 
\begin{equation}\label{qmuts}
	f^{\mu}(s,t)=\frac{te^{-\frac{(t-\mu s)^{2}}{2s}}}{\sqrt{2\pi s^{3}}}, \ s>0, \,t>0.
\end{equation}
It satisfies the following partial differential
equation (see D'Ovidio and Orsingher (2011), Eq. (4.4)):
\begin{equation}\label{fmu}
	\frac{\partial^{2}}{\partial t^{2}}f^{\mu}(s,t)-2\mu\frac{\partial}{\partial t}f^{\mu}(s,t)=2\frac{\partial}{\partial s}f^{\mu}(s,t).
\end{equation}

Let us consider the following time-changed process:
\begin{equation}\label{3.12}
	\hat{M}^{\mu}(t)=M(Z^{\mu}(t)),  \ t>0,
\end{equation}
 where the $\{M(t)\}_{t>0}$ is independent of $\{Z^{\mu}(t)\}_{t>0}$.

Next, we obtain the system of differential equations that governs the state probabilities  $\hat{p^{\mu}}(n,t)=\mathrm{Pr}\{\hat{M}^{\mu}(t)=n\}$, $n\ge0$.
\begin{theorem}
	The	pmf of $\{\hat{M}^{\mu}(t)\}_{t>0}$ solves
	\begin{equation*}
		\frac{\mathrm{d}^{2}}{\mathrm{d}t^{2}}\hat{p^{\mu}}(n,t)-2\mu\frac{\mathrm{d}}{\mathrm{d}t}\hat{p^{\mu}}(n,t)=2\bigg(\Lambda\hat{p^{\mu}}(n,t)-\sum_{j=1}^{n\wedge k}\lambda_{j}\hat{p^{\mu}}(n-j,t)\bigg),\ n\ge0.
	\end{equation*}
\end{theorem}
\begin{proof}
From \eqref{qmuts} and \eqref{3.12}, we have
\begin{equation}\label{pmu}
	\hat{p^{\mu}}(n,t)=\int_{0}^{\infty}p(n,s)f^{\mu}(s,t)\,\mathrm{d}s.
\end{equation}
On taking the second order derivative in \eqref{pmu}, we get
	\begin{align}
		\frac{\mathrm{d}^{2}}{\mathrm{d}t^{2}}\hat{p^{\mu}}(n,t)&=\int_{0}^{\infty}p(n,s)\frac{\partial^{2}}{\partial t^{2}}f^{\mu}(s,t)\,\mathrm{d}s\nonumber\\
		&=2\int_{0}^{\infty}p(n,s)\left(\mu\frac{\partial}{\partial t}f^{\mu}(s,t)+\frac{\partial}{\partial s}f^{\mu}(s,t)\right)\,\mathrm{d}s,\, \,  (\text{using}\, \eqref{fmu})\nonumber\\
		&=2p(n,s)f^{\mu}(s,t)\big|_{s=0}^{\infty}-2\int_{0}^{\infty}f^{\mu}(s,t)\frac{\mathrm{d}}{\mathrm{d}s}p(n,s)\,\mathrm{d}s+2\mu\frac{\mathrm{d}}{\mathrm{d}t}\hat{p^{\mu}}(n,t)\nonumber\\
		&=-2\int_{0}^{\infty}f^{\mu}(s,t)\bigg(-\Lambda p(n,s)+\sum_{j=1}^{n\wedge k}\lambda_{j}p(n-j,s)\bigg)\,\mathrm{d}s+2\mu\frac{\mathrm{d}}{\mathrm{d}t}\hat{p^{\mu}}(n,t),\nonumber
	\end{align}	
where the last step follows on using \eqref{gcpdE}. The proof is complete on using \eqref{pmu}.
\end{proof}
Next, we obtain the state probabilities of $\{\hat{M}^{\mu}(t)\}_{t>0}$. On using \eqref{p(n,t)} and \eqref{qmuts} in (\ref{pmu}), we get
\begin{align*}
	\hat{p^{\mu}}(n,t)
	&=\int_{0}^{\infty}\sum_{\Omega(k,n)}\Big(\prod_{j=1}^{k}\frac{\lambda_j^{x_j}}{x_j!}\Big)\frac{t}{\sqrt{2\pi}}s^{z_k-\frac{3}{2}}e^{-\Lambda s}e^{-\frac{(t-\mu s)^2}{2s}}\,\mathrm{d}s \\
	&=\sum_{\Omega(k,n)}\Big(\prod_{j=1}^{k}\frac{\lambda_j^{x_j}}{x_j!}\Big)\frac{\sqrt{2}t}{\sqrt{\pi}}e^{\mu t}\left(\frac{t^2}{2\Lambda+\mu^2}\right)^{\frac{z_k}{2}-\frac{1}{4}}K_{z_k-\frac{1}{2}}\left(t\sqrt{2\Lambda+\mu^2}\right),
\end{align*}
where in the last step we have used \eqref{rule}.

Also, the pgf of $\{\hat{M}^{\mu}(t)\}_{t>0}$ is given by
\begin{align}
\hat{G}^\mu(u,t)&=\int_{0}^{\infty}G(u,s)f^\mu(t,s)\,\mathrm{d}s\nonumber\\
&=e^{\mu t}\int_{0}^{\infty}e^{\big(-\frac{\mu^2}{2}-\sum_{j=1}^{k}\lambda_j(1-u^j)\big)s}\frac{te^{\frac{-t^2}{2s}}}{\sqrt{2\pi s^3}}\,\mathrm{d}s,\, \,  (\text{using}\, (\ref{pgfmt})\, \text{and}\, (\ref{qmuts}))\nonumber\\ 
&=\exp\Big(\mu t-t\Big(\mu^2+2\sum_{j=1}^{k}\lambda_j(1-u^j)\Big)^{1/2}\Big),\label{pgmmuet}
\end{align}
where in the last step we have used \eqref{lqts}. 

From \eqref{pgmmuet},
the mean and variance can be obtained in the following form: 
\begin{align*}
\mathbb{E}(\hat{M}^{\mu}(t))&=\begin{cases}
	\infty, \,\, \mu\leq 0,\vspace{.1cm}\\
	\sum_{j=1}^{k}j\lambda_{j}\frac{t}{\mu}, \,\, \mu>0,
\end{cases}\\
\operatorname{Var}(\hat{M}^{\mu}(t))&= \Big(\sum_{j=1}^{k}j^2\lambda_j +\Big(\frac{\sum_{j=1}^{k}j\lambda_j}{\mu}\Big)^2\Big)\frac{t}{\mu}, \,\, \mu>0.
\end{align*}

\begin{remark}
	For $k=1$, the results obtained in Section \ref{secfirstpas} reduces to the corresponding results for the time-changed Poisson processes (see Beghin and Orsingher (2012)).
\end{remark}

\subsection{GCP at Bessel times}\label{section4.2}
Let us consider a squared Bessel process  $R(t)=\left(X_{\gamma}(t)\right)^{2}$, $t\geq0$, where $\{X_{\gamma}(t)\}_{t\geq0}$ is a Bessel process of parameter $\gamma>0$ such that $X_{\gamma}(0)=0$. The squared Bessel process is a non-negative diffusion
process whose density function  is given by (see D'Ovidio and Orsingher (2011))
\begin{equation}\label{p2}
	f^{2}_{\gamma}(s,t)=\frac{e^{-s/2t}s^{\frac{\gamma}{2}-1}}{(2t)^{\frac{\gamma}{2}}\Gamma(\frac{\gamma}{2})}, \  s>0, \, t>0.
\end{equation}

Here, we consider the composition of GCP with an independent squared Bessel process, that is,
\begin{equation}\label{combes}
	\bar{M}^{\gamma}(t)= M(R(t)),\ t>0.
\end{equation} 
\begin{proposition}
	The state probabilities $\bar{p}_{\gamma}(n,t)=\mathrm{Pr}\{\bar{M}^{\gamma}(t)=n\}$  are given by
	\begin{equation*}
		\bar{p}_{\gamma}(n,t)=
		\sum_{\Omega(k,n)}\Big(\prod_{j=1}^{k}\frac{\lambda_{j}^{x_{j}}}{x_{j}!}\Big)\frac{(2t)^{z_{k}}\Gamma(z_{k}+\frac{\gamma}{2})}{\Gamma(\frac{\gamma}{2})(2\Lambda t+1)^{z_{k}+\frac{\gamma}{2}}},\  n\ge0,
	\end{equation*}
	where $z_{k}=x_{1}+x_{2}+\dots+x_{k}$.
\end{proposition}
\begin{proof}
	From \eqref{combes}, we have
	\begin{align*}
		\bar{p}_{\gamma}(n,t)&=\int_{0}^{\infty}p(n,s)	f^{2}_{\gamma}(s,t)\,\mathrm{d}s\\
		&=\sum_{\Omega(k,n)}\Big(\prod_{j=1}^{k}\frac{\lambda_{j}^{x_{j}}}{x_{j}!}\Big)\frac{\Gamma(z_k+\frac{\gamma}{2})}{(2t)^{\frac{\gamma}{2}}\Gamma(\frac{\gamma}{2})}\int_{0}^{\infty}e^{-\big(\Lambda +\frac{1}{2t}\big)s}s^{z_{k}+\frac{\gamma}{2}-1}\,\mathrm{d}s,
	\end{align*}
where we have used \eqref{p(n,t)} and \eqref{p2}.	This completes the proof.
\end{proof}
The pgf of $\{\bar{M}^\gamma(t)\}_{t>0}$ can be obtained as follows: 
\begin{align*}
	\bar{G}(u,t)&=\int_{0}^{\infty}G(u,s)f_{\gamma}^{2}(s,t)\,\mathrm{d}s\\
	&=\frac{1}{(2t)^\frac{\gamma}{2}\Gamma{(\frac{\gamma}{2})}}\int_{0}^{\infty}s^{\frac{\gamma}{2}-1}e^{-s\left(\frac{1}{2t}+\sum_{j=1}^{k}\lambda_{j}(1-u^j)\right)}\,\mathrm{d}s, \  (\text{using} \, (\ref{pgfmt}) \, \text{and}\, \eqref{p2} )\\
	&=\Big(1+2t\sum_{j=1}^{k}\lambda_j(1-u^j)\Big)^{-\frac{\gamma}{2}}.
\end{align*}
Also, its mean, second order factorial moment and variance can be obtained in the following form:
\begin{align*}
	\mathbb{E}(\bar{M}^\gamma(t))&=\frac{\partial}{\partial u}\bar{G}(u,t)\Big|_{u=1}=\sum_{j=1}^{k}j\lambda_j \gamma t,\\
	\mathbb{E}\big(\bar{M}^\gamma(t)(\bar{M}^\gamma(t)-1)\big)&=\frac{\partial^2}{\partial u^2}\bar{G}(u,t)\Big|_{u=1}=\sum_{j=1}^{k}j(j-1)\lambda_j \gamma t+\bigg(\sum_{j=1}^{k}j\lambda_jt\bigg)^2\gamma(\gamma+2)
\end{align*}
and 
\begin{equation*}
	\operatorname{Var}(\bar{M}^\gamma(t))	=\sum_{j=1}^{k}j^2\lambda_j\gamma t+2\gamma\bigg(\sum_{j=1}^{k}j\lambda_jt\bigg)^2,
\end{equation*}
respectively.
\begin{remark}
For $k=1$, the results obtained in Section \ref{section4.2} reduces to the corresponding results for the time-changed Poisson processes (see Beghin and Orsingher (2012), Section 5).
\end{remark}
\subsection{GCP time-changed by elastic Brownian motion}
Let $\{B_{\gamma}^{el}(t)\}_{t>0}$ be an elastic Brownian motion with absorbing rate $\gamma>0$. Its density function is given by (see Beghin and Orsingher (2012), Eq. (2.1))
\begin{equation*}
	f^{el}_{\gamma}(s,t) = 2e^{\gamma s}\int_{s}^{\infty} v e^{-\gamma v} \frac{e^{-\frac{v^{2}}{2t}}}{\sqrt{2 \pi t^{3}}}\,\mathrm{d}v + q_{\gamma}(t)\delta(s),
\end{equation*}
where $\delta(s)$ is Dirac's delta function and $q_{\gamma}(t)$ is the probability that the process is absorbed by the barrier in zero up to time $t$ which is given by
\begin{equation*}
	q_{\gamma}(t) = 1-\mathrm{Pr}\{B_{\gamma}^{el} (t) > 0\} = 1- 2e^{\frac{\gamma^2 t}{2}} \int_{\gamma \sqrt{t}}^{\infty} \frac{e^{-\frac{v^2}{2}}}{\sqrt{2\pi}}\,\mathrm{d}v.
\end{equation*}
For more details on elastic Brownian motion, we refer the reader to Beghin and Orsingher (2009). 

Here, we consider the following time-changed process:
\begin{equation}\label{pyegd}
	M^{el}_{\gamma}(t)=M(B_{\gamma}^{el}(t)),\  t>0,
\end{equation} 
where the GCP is independent of the elastic Brownian motion.

\begin{theorem}
For $\Lambda \neq \gamma$, the pmf $p^{el}_{\gamma}(n,t)=\mathrm{Pr}\{M^{el}_{\gamma}(t))=n\}$ is given by 
\begin{equation*}
p^{el}_{\gamma}(0,t)= 1 - \frac{\Lambda - \gamma -1}{\Lambda - \gamma} \mathrm{Pr}\{B_{\gamma}^{el}(t)>0\}- \frac{1}{\Lambda-\gamma}\mathrm{Pr}\{ B_{\Lambda}^{el}(t)>0\}
\end{equation*}
and 
\begin{align*}
	p^{el}_{\gamma}(n,t)
	&=\sum_{\Omega(k,n)}\Big(\prod_{j=1}^{k} \frac{\lambda_{j}^{x_{j}}}{x_{j}!}\Big) \frac{z_k!}{(\Lambda-\gamma)^{z_k+1} } \mathrm{Pr}\{B_{\gamma}^{el} (t)>0\} \\
	&\qquad\quad -\sum_{\Omega(k,n)}\Big(\prod_{j=1}^{k} \frac{\lambda_{j}^{x_{j}}}{x_{j}!}\Big) \frac{z_k!}{(\Lambda - \gamma)^{z_k+1}} \sum_{r=0}^{z_{k}}\frac{(\gamma - \Lambda)^{r}}{r!} \frac{\mathrm{d}^{r}}{\mathrm{d}\Lambda^{r}} \mathrm{Pr}\{B_{\Lambda}^{el} (t)>0\},\ n\geq 1.
\end{align*}
\begin{proof}
	From \eqref{pyegd}, we have
	\begin{equation*}
		p^{el}_{\gamma}(n,t)=\int_{0}^{\infty} \sum_{\Omega(k,n)}\Big(\prod_{j=1}^{k} \frac{(\lambda_{j}s)^{x_{j}}}{x_{j}!}\Big)e^{-\Lambda s}\bigg(2e^{\gamma s} \int_{s}^{\infty}v e^{-\gamma v} \frac{e^{\frac{-v^{2}}{2t}}}{\sqrt{2 \pi t^{3}}} \mathrm{d}v + q_{\gamma}(t) \delta(s)\bigg)\mathrm{d}s.
	\end{equation*}
Now, the proof follows similar lines to that of Theorem 2.1 of Beghin and Orsingher (2012). Thus, the details are omitted. 
\end{proof}
\end{theorem}
\begin{remark}
	Alternatively, the pmf of $\{M^{el}_{\gamma}(t)\}_{t>0}$ can be written in the following form: 
	\begin{equation*}
		p^{el}_{\gamma}(n,t)=\begin{cases*}
			\displaystyle 1-E_{\frac{1}{2},1}\Big(-\frac{\gamma\sqrt{t}}{\sqrt{2}}\Big)+\sum_{m=0}^{\infty}\Big(-\frac{\gamma\sqrt{t}}{\sqrt{2}}\Big)^{m}E_{\frac{1}{2},\frac{m}{2}+1}\Big(-\frac{\Lambda\sqrt{t}}{\sqrt{2}}\Big),\  n=0,\\
			\displaystyle \sum_{\Omega(k,n)}\Big(\prod_{j=1}^{k}\frac{\lambda_{j}^{x_{j}}}
			{x_{j}!}\Big)\frac{z_{k}!}{\sqrt{2^{z_{k}}}}\sum_{r=0}^{\infty}\Big(\frac{-\gamma\sqrt{t}}{\sqrt{2}}\Big)^{r}E_{\frac{1}{2},\frac{r+z_{k}}{2}+1}^{z_{k}+1}\Big(-\frac{\Lambda\sqrt{t}}{\sqrt{2}}\Big),\  n\ge1,
		\end{cases*}
	\end{equation*}
where $E_{\alpha,\beta}^\gamma(\cdot)$ denotes the three parameter Mittag-Leffler function defined in \eqref{mitag}.
This can be proved using Theorem 2.2 of Beghin and Orsingher (2012).
\end{remark}
The following result can be proved same as Theorem 2.3 of Beghin and Orsingher (2012).
\begin{theorem}
For $\Lambda=\gamma$, the pmf of $\{M^{el}_{\gamma}(t)\}_{t>0}$ is given by 	
\begin{equation*}
	p^{el}_{\gamma}(0,t)= 1 - \frac{\Lambda \sqrt{t}}{\sqrt{2}} E_{\frac{1}{2},\frac{1}{2}}^{2}\Big(\frac{-\Lambda\sqrt{t}}{\sqrt{2}}\Big)
\end{equation*}
	and 
	\begin{align*}
		p^{el}_{\gamma}(n,t)
		&=\sum_{\Omega(k,n)}\Big(\prod_{j=1}^{k} \frac{\lambda_{j}^{x_{j}}}{x_{j}!}\Big)z_k! \Big(\frac{t}{2}\Big)^{z_k/2}E_{\frac{1}{2},\frac{z_k}{2}+1}^{z_k+2}\Big(-\frac{\Lambda\sqrt{t}}{\sqrt{2}}\Big),\ n\geq 1.
	\end{align*}
\end{theorem}	
\subsection{GCP at Brownian sojourn times}
Let $\{B(t)\}_{t>0}$ be a standard Brownian motion and $S_{t}^{+}=meas\{s<t:B(s)>0\}$ be the Brownian sojourn time on the positive half-line. The density function of  $\{S_{t}^{+}\}_{t>0}$ is given by (see Beghin and Orsingher (2012), Eq. (4.2))
\begin{equation}\label{tauts}
	f^{+}(x,t)=\frac{1}{\pi\sqrt{x(t-x)}}, \ 0<x<t.
\end{equation}

 Here, we consider the composition of GCP with  independent Brownian sojourn time, that is,
\begin{equation*}
	M^{+}(t)=M(S_{t}^{+}), \ t>0.
\end{equation*}
Its pmf  $p^{+}(n,t)=\mathrm{Pr}\{M^{+}(t)=n\}$, $n\ge0$ can be obtained as follows:
\begin{align}
	p^{+}(n,t)&=\int_{0}^{t}p(n,x)f^{+}(x,t)\mathrm{d}x\nonumber\\
	&=\int_{0}^{t}\sum_{\Omega(k,n)}\Big(\prod_{j=1}^{k}\frac{(\lambda_{j}x)^{x_{j}}}{x_{j}!}\Big)e^{-\Lambda x}\frac{1}{\pi\sqrt{x(t-x)}}\,\mathrm{d}x.\label{sojude}
\end{align}
By using the following identity (see  Gradshteyn and Ryzhik (2007), p. 347):
\begin{equation}\label{6.3}
	\int_{0}^{u}x^{\alpha-1}(u-x)^{\gamma-1}e^{\beta x}\mathrm{d}x=B(\gamma,\mu)u^{\alpha+\gamma-1} \ _{1}F_{1}(\alpha;\alpha+\gamma;\beta u), \ \alpha>0,\ \gamma>0,\  \beta>0,
\end{equation}
where $_{1}F_{1}(\beta;\eta;x)$ is the confluent hypergeometric function defined in \eqref{confluent}, we have
\begin{align*}
	p^{+}(n,t)&=\sum_{\Omega(k,n)}\Big(\prod_{j=1}^{k}\frac{\lambda_{j}^{x_{j}}}{x_{j}!}\Big)\frac{t^{z_{k}}}{\pi}B\left(\frac{1}{2},z_{k}+\frac{1}{2}\right) \ _{1}F_{1}\left(z_k+\frac{1}{2};z_{k}+1;-\Lambda t\right)\nonumber\\
	&=\sum_{\Omega(k,n)}\Big(\prod_{j=1}^{k}\frac{\lambda_{j}^{x_{j}}}{x_{j}!}\Big)\frac{t^{z_{k}}}{\pi}e^{-\Lambda t}B\left(\frac{1}{2},z_{k}+\frac{1}{2}\right) \ _{1}F_{1}\left(\frac{1}{2};z_{k}+1;\Lambda t\right),
\end{align*}
where in the penultimate step we have used an identity given in Section 9.212 on p. 1023 of Gradshteyn and Ryzhik (2007).

The pmf given in \eqref{sojude} can be written as 
\begin{equation*}
	p^{+}(n,t)=\sum_{\Omega(k,n)}\Big(\prod_{j=1}^{k}\frac{(\lambda_{j}t)^{x_{j}}}{x_{j}!}\Big)\frac{1}{\pi}\int_{0}^{1}z^{z_{k}-\frac{1}{2}}(1-z)^{-\frac{1}{2}}e^{-\Lambda tz}\,\mathrm{d}z.
\end{equation*}
Its Laplace transform can be obtained as follows: 
\begin{align}
	\mathcal{L}(p^{+}(n,t); \eta)&=\int_{0}^{\infty}e^{-\eta t}\sum_{\Omega(k,n)}\Big(\prod_{j=1}^{k}\frac{(\lambda_{j}t)^{x_{j}}}{x_{j}!}\Big)\frac{1}{\pi}\int_{0}^{1}z^{z_{k}-\frac{1}{2}}(1-z)^{-\frac{1}{2}}e^{-\Lambda tz}\,\mathrm{d}z \,\mathrm{d}t\nonumber\\
	&=\sum_{\Omega(k,n)}\Big(\prod_{j=1}^{k}\frac{\lambda_{j}^{x_{j}}}{x_{j}!}\Big)\frac{1}{\pi}\int_{0}^{1}z^{z_{k}-\frac{1}{2}}(1-z)^{-\frac{1}{2}}\int_{0}^{\infty}t^{z_{k}}e^{-t(\Lambda z+\eta)}\,\mathrm{d}t\,\mathrm{d}z\nonumber\\
	&=\sum_{\Omega(k,n)}\Big(\prod_{j=1}^{k}\frac{\lambda_{j}^{x_{j}}}{x_{j}!}\Big)\frac{1}{\pi}\int_{0}^{1}z^{z_{k}-\frac{1}{2}}(1-z)^{-\frac{1}{2}}\frac{\Gamma(z_{k}+1)}{(\Lambda z+\eta)^{z_{k}+1}}\,\mathrm{d}z\nonumber\\
	&=\frac{1}{\pi}\sum_{\Omega(k,n)}\Big(\prod_{j=1}^{k}\frac{{\lambda_{j}}^{x_j}}{x_j!}\Big)\eta^{-(z_k+1)}z_k!\int_{0}^{1}z^{z_k-\frac{1}{2}}(1-z)^{-\frac{1}{2}}\bigg(1+\frac{\Lambda z}{\eta}\bigg)^{-(z_k+1)}\mathrm{d}z\nonumber\\
	&=\frac{1}{\pi}\sum_{\Omega(k,n)}\Big(\prod_{j=1}^{k}\frac{{\lambda_{j}}^{x_j}}{x_j!}\Big)\eta^{-(z_k+1)}z_k!\int_{0}^{1}z^{z_k-\frac{1}{2}}(1-z)^{-\frac{1}{2}}\sum_{r=0}^{\infty}\binom{z_k+r}{r}\bigg(\frac{-\Lambda z}{\eta}\bigg)^r\mathrm{d}z\nonumber\\
	&=\sqrt{\pi}\sum_{\Omega(k,n)}\Big(\prod_{j=1}^{k}\frac{{\lambda_{j}}^{x_j}}{x_j!}\Big)\eta^{-(z_k+1)}z_k!\sum_{r=0}^{\infty}\binom{z_k+r}{r}\bigg(\frac{-\Lambda}{\eta}\bigg)^{r}\frac{\Gamma{(z_k+r+\frac{1}{2})}}{\Gamma{(z_k+r+1)}}.\nonumber
\end{align}

The pgf $G^{+}(u,t)=\mathbb{E}(u^{M^{+}(t)})$, $|u|\leq1$ can be obtained as follows:
\begin{align}\label{int+}
	G^{+}(u,t)&=\int_{0}^{t}G(u,x)f^{+}(x,t)\,\mathrm{d}x\nonumber\\
	&=\int_{0}^{t}\frac{e^{-x\big(\sum_{j=1}^{k}\lambda_{j}(1-u^j)\big)}}{\pi\sqrt{x(t-x)}}\,\mathrm{d}x, \ (\text{using} \, (\ref{pgfmt}) \, \text{and}\, (\ref{tauts})) \\
	&=\, _1 F_1\Big(\frac{1}{2};1;\sum_{j=1}^{k}\lambda_{j}(u^j-1)t\Big),\  (\text{using}\,\eqref{6.3})\nonumber\\
	&=e^{\frac{1}{2}\sum_{j=1}^{k}\lambda_{j}(u^j-1)t}J_0\Big(\frac{1}{2}\sum_{j=1}^{k}\lambda_{j}(u^j-1)te^{\frac{\pi}{2}i}\Big),\nonumber
\end{align}
where the last step follows from an identity given in Section 9.215 on  p. 1024 of Gradshteyn and Ryzhik (2007). Here, $J_0$ is the Bessel function of order $0$.

On substituting $x=t\sin^{2}\phi$ in \eqref{int+}, we get
\begin{align*}
	G^{+}(u,t)&=\frac{2}{\pi}\int_{0}^{\pi/2}e^{t\sin^{2}\phi\sum_{j=1}^{k}\lambda_{j}(u^j-1)}\mathrm{d}\phi\\
	&=\frac{2}{\pi} e^{\sum_{j=1}^{k}\frac{\lambda_{j}}{2}(u^j-1)t}\int_{0}^{\pi/2}e^{\sum_{j=1}^{k}\frac{\lambda_{j}}{2}(u^j-1)t\cos2\phi}\mathrm{d}\phi,\  (\text{as} \ \sin^{2}\phi=(1-\cos2\phi)/2)\\
	&=\frac{1}{\pi}e^{\sum_{j=1}^{k}\frac{\lambda_{j}}{2}(u^j-1)t}\int_{0}^{\pi}e^{\sum_{j=1}^{k}\frac{\lambda_{j}}{2}(u^j-1)t\cos\theta}\mathrm{d}\theta\\
	&=e^{\sum_{j=1}^{k}\frac{\lambda_{j}}{2}(u^j-1)t}I_{0}\Big(\sum_{j=1}^{k}\frac{\lambda_{j}}{2}(u^j-1)t\Big)\\
	&=G^{*}(u,t)I_{0}\Big(\sum_{j=1}^{k}\frac{\lambda_{j}}{2}(u^j-1)t\Big),
\end{align*}
where $G^{*}(u,t)$ is the pgf of GCP that performs independently $k$ kinds of jumps of amplitude $1,2,\cdots,k$ with positive rates $\lambda_{1}/2,\lambda_{2}/2,\dots, \lambda_{k}/2$. Here, $I_0$ is  the Bessel function of order $0$.

By using \eqref{int+}, the mean of $\{M^{+}(t)\}_{t>0}$ can be obtained as follows:
\begin{equation*}
	\mathbb{E}(M^{+}(t))=\frac{\partial}{\partial u}G^{+}(u,t)\Big|_{u=1}=\frac{1}{\pi}\sum_{j=1}^{k}j\lambda_j\int_{0}^{t}x^{1/2}(t-x)^{-1/2}\mathrm{d}x=\frac{1}{2\pi}\sum_{j=1}^{k}j\lambda_j t.
\end{equation*}
Also, its second factorial moment is 
\begin{align*}
	\mathbb{E}(M^{+}(t)(M^{+}(t)-1))&=\frac{\partial^2}{\partial u^2}G^{+}(u,t)\Big|_{u=1}\\
	&=\frac{1}{\pi}\Big(\sum_{j=1}^{k}j\lambda_j\Big)^2\int_{0}^{t}x^{3/2}(t-x)^{-1/2}\mathrm{d}x\\
	&\ \ +\frac{1}{\pi}\sum_{j=1}^{k}j(j-1)\lambda_j\int_{0}^{t}x^{1/2}(t-x)^{-1/2}\mathrm{d}x\\
	&=\frac{3}{8}\Big(\sum_{j=1}^{k}j\lambda_j t\Big)^2+\sum_{j=1}^{k}j(j-1)\lambda_j\frac{t}{2}
\end{align*}
and its variance is given by
$\operatorname{Var}(M^{+}(t))=\frac{1}{8}\big(\sum_{j=1}^{k}j\lambda_j t\big)^2+\frac{1}{2}\sum_{j=1}^{k}j^2\lambda_j t.
$

\section{GCP time changed by subordinators linked to incomplete gamma function}\label{sec3}
Beghin and Ricciuti (2021) defined a subordinator $\{G_{\alpha}(t)\}_{t\geq0}$, $0<\alpha\le 1$, as a non-decreasing L\'evy process whose Laplace exponent is given by $\phi(\eta)=\alpha\gamma(\alpha;\eta)$, $\eta>0$ where $\gamma(\alpha; \eta)$ is the lower-incomplete gamma function defined as follows:
\begin{equation*}
	\gamma(\alpha;\eta)\coloneqq\int_{0}^{\eta}e^{-x}x^{\alpha-1}\, \mathrm{d}x.
\end{equation*} 
For $\alpha=1$, the subordinator $\{G_{\alpha}(t)\}_{t\geq0}$ reduces to the standard Poisson process and for $\alpha\in (0,1)$, it can be represented as a compound Poisson process with positive jumps of size greater than one. Also, they introduced a generalization of $\{G_{\alpha}(t)\}_{t\geq0}$ whose jumps sizes are greater than or equal to $\epsilon>0$. Here, we denote this generalized subordinator by $\{G_{\alpha}^{(\epsilon)}(t)\}_{t\geq0}$ whose  Laplace exponent is given by $\phi^{(\epsilon)}(\eta)=\alpha\epsilon^{-\alpha}\gamma(\alpha;\eta \epsilon)$. The subordinator $\{G_{\alpha}(t)\}_{t\geq0}$ has no finite moments of any integer order. To overcome this limitation, they introduced a tempered version of it, denoted by  $\{G_{\alpha,\theta}(t)\}_{t\geq0}$, with the following Laplace exponent:
\begin{equation*}
\phi_{\alpha,\theta}(\eta)=\alpha\gamma(\alpha;\eta+\theta)-\alpha\gamma(\alpha;\theta),\ \theta>0,
\end{equation*}
 is the tempering parameter. It exhibits finite moment
of any integer order.

In this section, we study time-changed GCPs using the subordinators $\{G_{\alpha}^{(\epsilon)}(t)\}_{t\geq0}$ and $\{G_{\alpha, \theta}(t)\}_{t\geq0}$.

\subsection{GCP time-changed by $\{G_{\alpha}^{(\epsilon)}(t)\}_{t\geq0}$}

Let us consider the following time-changed process:
\begin{equation*}
	M_\alpha^{\epsilon}(t)=M(G_\alpha^{(\epsilon)}(t)),\  t\ge0,
\end{equation*}
where the GCP $\{M(t)\}_{t\ge0}$ is independent of $\{G_{\alpha}^{(\epsilon)}(t)\}_{t\geq0}$.

\begin{remark}
	For $\epsilon =1$, the process $\{M_\alpha^\epsilon(t)\}_{t\geq0}$ reduces to the GCP time-changed by subordinator $\{G_\alpha(t)\}_{t\geq0}$. 
\end{remark}

The Laplace transform of $\{M_\alpha^\epsilon(t)\}_{t\geq0}$ is given by 
\begin{align}
	\mathbb{E}(e^{-s M_\alpha^{\epsilon}(t)})&=\mathbb{E}\left(\mathbb{E}\left(e^{-s M_\alpha^{\epsilon}(t)}\big|G_\alpha^{(\epsilon)}(t)\right)\right),\  s>0\nonumber\\
	&=\mathbb{E}\Big(\exp\Big(-\sum_{j=1}^{k}\lambda_j(1-e^{-s j})G_\alpha^{(\epsilon)} (t)\Big)\Big),\ (\text{using} \ \eqref{mgfmt})\nonumber\\
	&=\exp\Big(-\alpha t\epsilon^{-\alpha}\gamma\Big(\alpha;\sum_{j=1}^{k}\lambda_j(1-e^{-s j})\epsilon\Big)\Big),\label{chargamma}
\end{align}
where in the last step we have used the Laplace transform of $\{G_{\alpha}^{(\epsilon)}(t)\}_{t\geq0}$ (see Beghin and Ricciuti (2021), Section 3.2.1). Similarly, its pgf can be obtained as follows:
\begin{align*}
	\mathbb{E}(u^{M_\alpha^{\epsilon}(t)})&=\mathbb{E}\big(\mathbb{E}(u^{M_\alpha^{\epsilon}(t)}\big|G_\alpha^{(\epsilon)}(t))\big),\  |u|\le 1\\
	&=\mathbb{E}\Big(\exp\Big(-\sum_{j=1}^{k}\lambda_j(1-u^j)G_\alpha^{(\epsilon)}(t)\Big)\Big),\  (\text{using}\ \eqref{pgfmt} )\\
	&=\exp\Big(-\alpha t\epsilon^{-\alpha}\gamma\Big(\alpha;\sum_{j=1}^{k}\lambda_j(1-u^j)\epsilon\Big)\Big).
\end{align*}
% On taking $k=1$, the process $\{	M_\alpha^{\epsilon}(t)\}_{t\ge0}$ reduces to the Poisson process time-changed by $\{G_{\alpha}^{(\epsilon)}(t)\}_{t\geq0}$ (see Babulal {\it et al.} (2024)).
 
 Let $X$ be a discrete random variable with support $\mathbb{N}_{0}$. The following holds true:
\begin{equation}\label{usi}
\mathrm{Pr}\{X=n\}=\frac{1}{n!}\dfrac{\mathrm{d}^n}{\mathrm{d}u^n} \mathbb{E}(u^{X})\Big|_{u=0}, \ n=0,1,2,\dots.
\end{equation}
By using \eqref{usi} for $\{M_\alpha^{\epsilon}(t)\}_{t\ge0}$, we get
\begin{align*}
	\mathrm{Pr}\{M_\alpha^\epsilon(t)=0\}&=e^{-\alpha t\epsilon^{-\alpha}\gamma(\alpha;\Lambda\epsilon)},\\
	\mathrm{Pr}\{M_\alpha^\epsilon(t)=1\}&=\lambda_1\alpha t\Lambda^{\alpha -1}e^{-\alpha t\epsilon^{-\alpha}\gamma(\alpha;\Lambda\epsilon)-\Lambda\epsilon)},\\
	\mathrm{Pr}\{M_\alpha^\epsilon(t)=2\}&= \Lambda^{\alpha-2}\frac{\alpha t}{2!}e^{-\alpha t\epsilon^{-\alpha}\gamma(\alpha;\Lambda\epsilon)-2\Lambda\epsilon}(\Lambda\lambda_1^2\epsilon e^{\Lambda\epsilon}-\lambda_1^2(\alpha -1)e^{\Lambda\epsilon}+2\lambda_2\Lambda e^{\Lambda\epsilon} +\lambda_1^2\Lambda^\alpha \alpha t),
\end{align*} 
and so on.
\begin{proposition}\label{asympinge}
	The following asymptotic result holds true for the tail probability of $\{M_\alpha^\epsilon(t)\}_{t\geq0}$: 
	\begin{equation*}
		\mathrm{Pr}\{M_\alpha^\epsilon(t)>y\}\sim \frac{tc_{1}^\alpha y^{-\alpha} }{\Gamma(1-\alpha)}, \  y\to\infty,
	\end{equation*}
	where $c_{1}=\lambda_1+2\lambda_2+\cdots+ k\lambda_k$.
\end{proposition}
\begin{proof}
	For $s>0$, we have
	\begin{align}\label{123}
		\int_{0}^{\infty}e^{-s y}\mathrm{Pr}\{M_\alpha^\epsilon(t)>y\}\, \mathrm{d}y&=\big(1-\mathbb{E}\big(e^{-s M_\alpha^\epsilon(t)}\big)\big)s^{-1}\nonumber\\
		&=\Big(1-\exp\Big(-\alpha t\epsilon^{-\alpha}\gamma\big(\alpha;\sum_{j=1}^{k}\lambda_j(1-e^{-s j})\epsilon\big)\Big)\Big)s^{-1},
	\end{align}
	where the last step follows from \eqref{chargamma}.
	On letting $s\to 0$ and using Taylor series expansion up to first order in \eqref{123}, we get 
{\small	\begin{align*}
	\Big(1-\exp\Big(-\alpha t\epsilon^{-\alpha}\gamma\big(\alpha;\sum_{j=1}^{k}\lambda_j(1-e^{-s j})\epsilon\big)\Big)\Big)s^{-1}&\sim	\Big(1-\Big(1-\alpha t\epsilon^{-\alpha}\gamma\big(\alpha;\sum_{j=1}^{k}\lambda_j(1-e^{-s j})\epsilon\big)\Big)\Big)s^{-1}\\
		&\sim  t\epsilon^{-\alpha}\Big(\sum_{j=1}^{k}\lambda_j(1-e^{-s j})\epsilon\Big)^\alpha s^{-1}\\
		&\sim tc_{1}^{\alpha}s^{\alpha -1},
	\end{align*}}
	where in the penultimate step we have used the following result (see Beghin and Ricciuti (2021), Eq. (18)):
	\begin{equation*}
		\gamma(\alpha;s)\sim \frac{s^{\alpha}}{\alpha}.
	\end{equation*}
	 The proof is complete by using the Tauberian theorem (see Feller (1971), Theorem 4, p. 446).
\end{proof}

\subsection{GCP time-changed by tempered subordinator $\{G_{\alpha,\theta}(t)\}_{t\geq0}$} 
Let $\{M(t)\}_{t\ge0}$ be a GCP that is independent of the tempered subordinator $\{G_{\alpha,\theta}(t)\}_{t\geq0}$. Consider the following time-changed GCP: 
\begin{equation}\label{matheta}
	M_{\alpha,\theta}(t)=M(G_{\alpha,\theta}(t)),\  t\ge0.
\end{equation}
Its Laplace transform can be obtained as follows:
\begin{align*}
	\mathbb{E}(e^{-s M_{\alpha, \theta}(t)})&=\mathbb{E}(\mathbb{E}(e^{-s M_{\alpha,\theta}(t)}|G_{\alpha,\theta}(t))),\ s>0\\
	&=\mathbb{E}\Big(\exp\Big(-\sum_{j=1}^{k}\lambda_j(1-e^{-s j})G_{\alpha, \theta} (t)\Big)\Big),\  (\text{using} \ \eqref{mgfmt})\\
	&=\exp\Big(-\alpha t\Big(\gamma\Big(\alpha;\sum_{j=1}^{k}\lambda_j(1-e^{-s j})+\theta\Big)-\gamma(\alpha;\theta)\Big)\Big),
\end{align*}
where in the last step we have used the Laplace transform of $\{G_{\alpha,\theta}(t)\}_{t\geq0}$ (see Beghin and Ricciuti (2021), Eq. (27)). Similarly, its pgf is given by
\begin{equation*}
	\mathbb{E}(u^{M_{\alpha, \theta}(t)})=\exp\Big(-\alpha t \Big(\gamma\Big(\alpha;\sum_{j=1}^{k}\lambda_j(1-u^j)+\theta\Big)-\gamma(\alpha;\theta)\Big)\Big),\  |u|\le 1.
\end{equation*}
By using \eqref{usi} for $\{M_{\alpha,\theta}(t)\}_{t\ge0}$, we get
\begin{align*}
	\mathrm{Pr}\{M_{\alpha, \theta}(t)=0\}&=\exp(-\alpha t\left(\gamma(\alpha;\Lambda+\theta)-\gamma(\alpha;\theta)\right)),\\
	\mathrm{Pr}\{M_{\alpha, \theta}(t)=1\}&=\lambda_1\alpha t(\Lambda+\theta)^{\alpha -1}\exp(-\alpha t\left(\gamma(\alpha;\Lambda+\theta)-\gamma(\alpha;\theta)\right)-(\Lambda+\theta)),\\
	\mathrm{Pr}\{M_{\alpha, \theta}(t)=2\}&=\frac{\alpha t}{2!}(\Lambda+\theta)^{\alpha-2}\exp(-\alpha t\gamma(\alpha;\Lambda+\theta)-\gamma(\alpha;\theta)-(\Lambda+\theta))(\lambda_1^2(\Lambda+\theta)\\
	&\hspace*{3cm}-(\alpha-1)\lambda_1^2+2\lambda_2(\Lambda+\theta)+e^{-(\Lambda+\theta)}\Lambda\lambda_1^2\alpha t(\Lambda+\theta )^{\alpha-1}),
\end{align*}
and so on.

Next, we obtain a closed form expression for the state probabilities of $\{M_{\alpha,\theta}(t)\}_{t\geq0}$. 

Let $f(\cdot,t)$ be the pdf of $\{G_{\alpha,\theta}(t)\}_{t\geq0}$. Then, on using \eqref{p(n,t)} and \eqref{matheta}, we get
\begin{align}\label{pmfat}
	\mathrm{Pr}\{M_{\alpha,\theta}(t)=n\}&=\int_{0}^{\infty}p(n,y) f(y,t)\, \mathrm{d}y\nonumber\\
	&=\int_{0}^{\infty}\sum_{\Omega(k,n)}\Big(\prod_{j=1}^{k}\frac{(\lambda_j y)^{x_j}}{x_j!}\Big)e^{-\Lambda y}f(y,t)\, \mathrm{d}y\nonumber\\
	&=\sum_{\Omega(k,n)}\Big(\prod_{j=1}^{k}\frac{\lambda_j ^{x_j}}{x_j!}\Big) \mathbb{E}\left(e^{-\Lambda G_{\alpha,\theta}(t)}(G_{\alpha,\theta}(t))^{z_k}\right), \ n\ge0,
\end{align}
where $z_k=x_1+ x_{2}+\cdots+x_{k}$. 

Note that
\begin{align*}
	\sum_{n=0}^{\infty}	\mathrm{Pr}\{M_{\alpha,\theta}(t)=n\}&=\sum_{n=0}^{\infty}\sum_{\Omega(k,n)}\Big(\prod_{j=1}^{k}\frac{\lambda_j ^{x_j}}{x_j!}\Big) \mathbb{E}\left(e^{-\Lambda G_{\alpha,\theta}(t)}(G_{\alpha,\theta}(t))^{z_k}\right)\\
	&=\int_{0}^{\infty}\sum_{n=0}^{\infty}\sum_{\Omega(k,n)}\Big(\prod_{j=1}^{k}\frac{(\lambda_j y)^{x_j}}{x_j!}\Big)e^{-\Lambda y}f(y,t)\, \mathrm{d}y\\
	&=\int_{0}^{\infty}f(y,t)\, \mathrm{d}y=1.
\end{align*}
Thus, \eqref{pmfat} is indeed a pmf.

On using Theorem 2.1 of Leonenko {\it et al.} (2014), Theorem 3 of Beghin and Riccuiti (2021) and \eqref{covgcp}, we get the mean, variance and covariance of $\{M_{\alpha,\theta}(t)\}_{t\geq0}$ in the following form:
\begin{align*}
	\mathbb{E}(M_{\alpha,\theta}(t))&=\mathbb{E}(M(1))\mathbb{E}(G_{\alpha,\theta}(t))=c_1\alpha t \theta^{\alpha -1}e^{-\theta},\\
	\operatorname{Var}(M_{\alpha,\theta}(t))&=(\mathbb{E}M(1))^2 \operatorname{Var}(G_{\alpha,\theta}(t))+ \mathbb{E}(G_{\alpha,\theta}(t))\operatorname{Var}(M(1))\\
	&=c_1^2(\alpha t\theta^{\alpha-1}e^{-\theta}+\alpha(1-\alpha)t\theta^{\alpha-2}e^{-\theta})+c_2\alpha t \theta^{\alpha-1}e^{-\theta},\\
	\operatorname{Cov}(M_{\alpha,\theta}(s),M_{\alpha,\theta}(t))&=(\mathbb{E}M(1))^2 \operatorname{Cov}(G_{\alpha,\theta}(s),G_{\alpha,\theta}(t))+ \mathbb{E}(G_{\alpha,\theta}(\min\{s,t\}))\operatorname{Var}(M(1))\\
	&=(\mathbb{E}M(1))^2 \operatorname{Var}(G_{\alpha,\theta}(s))+ \mathbb{E}(G_{\alpha,\theta}(s))\operatorname{Var}(M(1))\\
	&=c_1^2(\alpha s\theta^{\alpha -1}e^{-\theta}+\alpha(1-\alpha)s\theta^{\alpha -2}e^{-\theta})+c_2\alpha s\theta^{\alpha -1}e^{-\theta},\  0<s\le t.
\end{align*}

\begin{theorem}
	The process $\{M_{\alpha,\theta}(t)\}_{t\geq0}$ has LRD property.
\end{theorem}
\begin{proof}
	For fixed $s$ and large $t$, the correlation function of $\{M_{\alpha,\theta}(t)\}_{t\geq0}$ can be written as
	\begin{align*}
		\operatorname{Corr}(M_{\alpha,\theta}(s),M_{\alpha,\theta}(t))&=\frac{	\operatorname{Cov}\left(M_{\alpha,\theta}(s),M_{\alpha,\theta}(t)\right)}{\sqrt{\operatorname{Var}(M_{\alpha,\theta}(s))}\sqrt{\operatorname{Var}(M_{\alpha,\theta}(t))}}\\
		&\sim t^{-\frac{1}{2}}s^{\frac{1}{2}},
	\end{align*}
as $t\to \infty$. Thus, $\{M_{\alpha,\theta}(t)\}_{t\geq0}$ exhibits the LRD property.
\end{proof} 
Next, we obtain an asymptotic behavior of the tail probability of $\{M_{\alpha,\theta}(t)\}_{t\geq0}$. Its proof follows similar lines to that of Theorem $\ref{asympinge}$. Thus, it is omitted.
\begin{theorem}
	The following asymptotic result holds true for the tail probability of $\{M_{\alpha,\theta}(t)\}_{t\geq0}$:
	\begin{equation*}
		\mathrm{Pr}\{M_{\alpha,\theta}(t)>y\}\sim \frac{tc_{1}^{\alpha}y^{-\alpha} }{\Gamma(1-\alpha)}, \ y\to\infty.
	\end{equation*}
\end{theorem}
\begin{remark}
	On taking $k=1$ in the mean, variance, covariance of $\{M_{\alpha,\theta}(t)\}_{t\geq0}$, we get the corresponding results for the time-changed Poisson process (see Babulal {\it et al.} (2025), Section 5).
\end{remark}
\section{Fractional integral of the GFCP}\label{sec6}
Orsingher and Polito (2013) considered the Riemann-Liouville (R-L) fractional integral of the TFPP, that is, 
\begin{equation*}
	\mathcal{N}^{\alpha,\beta}(t)=\frac{1}{\Gamma(\alpha)}\int_{0}^{t}(t-s)^{\alpha-1}N^\beta(s)\,\mathrm{d}s, \ t\geq0, \  \alpha>0,
\end{equation*}
where $\{N^\beta(t)\}_{t\ge0}$ is a TFPP of order $0<\beta\leq1$. For $\beta=1$, it reduces to the fractional integral of the classical Poisson process with intensity $\lambda$. Recently, Vishwakarma and Kataria (2024) studied the integrals of birth-death processes at random times and derived asymptotic distributional properties for a particular case.

Here, we consider the R-L fractional integral of GFCP. It is defined as
\begin{equation}\label{malphabeta}
\mathcal{M}^{\alpha,\beta}(t)\coloneqq\frac{1}{\Gamma(\alpha)}\int_{0}^{t}(t-s)^{\alpha-1}M^\beta(s)\,\mathrm{d}s, \ t\geq0.
\end{equation}
For $k=1$, the process $\{\mathcal{M}^{\alpha,\beta}(t)\}_{t\ge0}$ reduces to the fractional integral of TFPP. The motivation to study this process lies in the fact that the integrated counting processes often arise in epidemic model, biological sciences, {\it etc} (see Jerwood (1970), Stefanov and Wang (2000), Pollett (2003) and the references therein). It is interesting to analyse the integrated process \eqref{malphabeta} since it allows to generalize the ideas behind such studies to a non-integer framework. For $\alpha=m\in\mathbb{N}$, the R-L fractional integral defined in \eqref{malphabeta} coincides with the following multiple integral:
\begin{equation*}
	\frac{1}{(m-1)!}\int_{0}^{t}(t-s)^{m-1}M^\beta(s)\,\mathrm{d}s=\int_{0}^{t}\mathrm{d}s_{1}\int_{0}^{s_{1}}\mathrm{d}s_{2}\cdots\int_{0}^{s_{m-1}}M^\beta(s_{m})\mathrm{d}s_{m}.
\end{equation*}
By using \eqref{meanvargfcp}, the mean of the fractional integral of GFCP is obtained as follows:
\begin{align}
	\mathbb{E}(\mathcal{M}^{\alpha,\beta}(t))&=\frac{1}{\Gamma(\alpha)}\int_{0}^{t}(t-s)^{\alpha-1}\mathbb{E}(M^\beta(s))\,\mathrm{d}s\nonumber\\
	&=\frac{c_1}{\Gamma(\alpha)\Gamma(\beta+1)}\int_{0}^{t}(t-s)^{\alpha-1}s^\beta\,\mathrm{d}s=\frac{c_{1}t^{\alpha+\beta}}{\Gamma(\alpha+\beta+1)}, \label{8.5}
\end{align}
where $c_{1}=\lambda_{1}+2\lambda_2+\dots+k\lambda_{k}$. Thus, from \eqref{meanvargfcp} and \eqref{8.5}, it follows that if $0<\alpha+\beta\leq1$ then  $\mathbb{E}(\mathcal{M}^{\alpha,\beta}(t))=\mathbb{E}(M^{\alpha+\beta}(t)).$

In the next result, we obtain the variance of the  fractional integral of GFCP. 
\begin{theorem}\label{varmabt}
The variance of	$\{\mathcal{M}^{\alpha,\beta}(t)\}_{t\ge0}$, $0<\beta<1$ has the following form:
{\small	\begin{align*}
		\operatorname{Var}(\mathcal{M}^{\alpha,\beta}(t))&=\frac{2c_{2}B(\beta+1,2\alpha)}{(\Gamma(\alpha))^{2}\Gamma(\beta+1)\alpha}t^{2\alpha+\beta}+\left(\frac{2B(2\beta+1,2\alpha)}{(\Gamma(\alpha))^{2}\Gamma(2\beta+1)\alpha}-\frac{1}{(\Gamma(\alpha+\beta+1))^{2}}\right)c_{1}^{2}t^{2\alpha+2\beta}\\
		&\ \ +\frac{2c_{1}^{2}\beta}{(\Gamma(\alpha)\Gamma(\beta+1))^{2}}\int_{0}^{t}\int_{s}^{t}(t-s)^{\alpha-1}(t-w)^{\alpha-1} w^{2\beta}B(\beta,\beta+1;s/w)\,\mathrm{d}w\,\mathrm{d}s,
	\end{align*}}
	where $c_{1}=\sum_{j=1}^{k}j\lambda_j$ and $c_{2}=\sum_{j=1}^{k}j^{2}\lambda_j$.
\end{theorem}
\begin{proof}
	On using a result given in Eq. (2.17) of Orsingher and Polito (2013) for GFCP, we get 
{\small	\begin{equation}\label{secmomfrac}
		\mathbb{E}\left(\frac{1}{\Gamma(\alpha)}\int_{0}^{t}(t-s)^{\alpha-1}M^\beta(s)\,\mathrm{d}s\right)^2=\frac{1}{(\Gamma(\alpha))^{2}}\int_{0}^{t}\int_{0}^{t}(t-s)^{\alpha-1}(t-w)^{\alpha-1}\mathbb{E}(M^\beta(s)M^\beta(w))\,\mathrm{d}s\,\mathrm{d}w.
	\end{equation}}
	On substituting \eqref{covmb} in the above equation, we get
	\begin{align}
		\mathbb{E}&\left(\frac{1}{\Gamma(\alpha)}\int_{0}^{t}(t-s)^{\alpha-1}M^\beta(s)\,\mathrm{d}s\right)^2\nonumber\\
		&=\frac{2}{(\Gamma(\alpha))^{2}}\int_{0}^{t}\int_{s}^{t}(t-s)^{\alpha-1}(t-w)^{\alpha-1}\bigg(\frac{c_{2}s^{\beta}}{\Gamma(\beta+1)}+\frac{c_{1}^{2}s^{2\beta}}{\Gamma(2\beta+1)}\nonumber\\
		&\hspace*{8cm} +\frac{c_{1}^{2}\beta w^{2\beta}B(\beta,\beta+1;s/w)}{(\Gamma(\beta+1))^{2}}\bigg)\,\mathrm{d}w\,\mathrm{d}s\nonumber\\
		&=\frac{2}{(\Gamma(\alpha))^{2}}\frac{c_{2}}{\Gamma(\beta+1)}\int_{0}^{t}s^{\beta}(t-s)^{\alpha-1}\mathrm{d}s\int_{s}^{t}(t-w)^{\alpha-1}\mathrm{d}w\nonumber\\
		&\ \ +\frac{2}{(\Gamma(\alpha))^{2}}\frac{c_{1}^{2}}{\Gamma(2\beta+1)}\int_{0}^{t}s^{2\beta}(t-s)^{\alpha-1}\mathrm{d}s\int_{s}^{t}(t-w)^{\alpha-1}\mathrm{d}w\nonumber\\
		&\ \ +\frac{2c_{1}^{2}\beta}{(\Gamma(\alpha)\Gamma(\beta+1))^{2}}\int_{0}^{t}\int_{s}^{t}(t-s)^{\alpha-1}(t-w)^{\alpha-1} w^{2\beta}B(\beta,\beta+1;s/w)\,\mathrm{d}w\,\mathrm{d}s\nonumber\\
		&=\frac{2}{(\Gamma(\alpha))^{2}}\frac{c_{2}}{\Gamma(\beta+1)\alpha}\int_{0}^{t}s^{\beta}(t-s)^{2\alpha-1}\mathrm{d}s+\frac{2}{(\Gamma(\alpha))^{2}}\frac{c_{1}^{2}}{\Gamma(2\beta+1)\alpha}\int_{0}^{t}s^{2\beta}(t-s)^{2\alpha-1}\mathrm{d}s\nonumber\\
		&\ \ +\frac{2c_{1}^{2}\beta}{(\Gamma(\alpha)\Gamma(\beta+1))^{2}}\int_{0}^{t}\int_{s}^{t}(t-s)^{\alpha-1}(t-w)^{\alpha-1} w^{2\beta}B(\beta,\beta+1;s/w)\,\mathrm{d}w\,\mathrm{d}s\nonumber\\
		&=\frac{2c_{2}B(\beta+1,2\alpha)}{(\Gamma(\alpha))^{2}\Gamma(\beta+1)\alpha}t^{2\alpha+\beta}+\frac{2c_{1}^{2}B(2\beta+1,2\alpha)}{(\Gamma(\alpha))^{2}\Gamma(2\beta+1)\alpha}t^{2\alpha+2\beta}\nonumber\\
		&\ \ +\frac{2c_{1}^{2}\beta}{(\Gamma(\alpha)\Gamma(\beta+1))^{2}}\int_{0}^{t}\int_{s}^{t}(t-s)^{\alpha-1}(t-w)^{\alpha-1} w^{2\beta}B(\beta,\beta+1;s/w)\,\mathrm{d}w\,\mathrm{d}s.\label{2ndordermoment}
	\end{align}
	Thus, the result follows on using \eqref{8.5} and \eqref{2ndordermoment}.

\end{proof}
On taking $\beta=1$ in \eqref{malphabeta}, we get the R-L fractional integral of GCP. It is given by 
\begin{equation}\label{fringc}
\mathcal{M}^\alpha(t)=\frac{1}{\Gamma(\alpha)}\int_{0}^{t}(t-s)^{\alpha-1}M(s)\,\mathrm{d}s,\  t\geq0.
\end{equation}
Thus, 
\begin{equation}\label{meanintgcp}
	\mathbb{E}(\mathcal{M}^\alpha(t))=\frac{c_{1}t^{\alpha+1}}{\Gamma(\alpha+2)}.
\end{equation}
Next, we obtain the variance of fractional integral of GCP whose proof follows along the similar lines to that of Theorem \ref{varmabt}.
\begin{theorem}\label{varmat}
The variance of $\{\mathcal{M}^\alpha(t)\}_{t\ge0}$ is 
\begin{equation*}
\operatorname{Var}(\mathcal{M}^\alpha(t))=\frac{c_2t^{2\alpha+1}}{(2\alpha+1)(\Gamma(\alpha+1))^{2}}.
\end{equation*}
\end{theorem}
\begin{proof}
	Using \eqref{covgcp} in \eqref{secmomfrac} for $\beta=1$, we get
	\begin{align}
		\mathbb{E}(\mathcal{M}^\alpha(t))^2&=\frac{2}{(\Gamma(\alpha))^{2}}\int_{0}^{t}\int_{s}^{t}(t-s)^{\alpha-1}(t-w)^{\alpha-1}(c_2s+c_1^{2}sw)\,\mathrm{d}w\,\mathrm{d}s\nonumber\\
		&=\frac{2c_2}{(\Gamma(\alpha))^{2}}\int_{0}^{t}s(t-s)^{\alpha-1}\,\mathrm{d}s\int_{s}^{t}(t-w)^{\alpha-1}\,\mathrm{d}w \nonumber\\ 
		&\hspace*{3cm}+\frac{2c_1^{2}}{(\Gamma(\alpha))^{2}}\int_{0}^{t}s(t-s)^{\alpha-1}\,\mathrm{d}s\int_{s}^{t}w(t-w)^{\alpha-1}\,\mathrm{d}w\nonumber\\
		&=\frac{2c_2}{(\Gamma(\alpha))^{2}}\frac{1}{\alpha}\int_{0}^{t}s(t-s)^{2\alpha-1}\,\mathrm{d}s  +\frac{2c_{1}^2}{(\Gamma(\alpha))^{2}}\frac{1}{\alpha}\int_{0}^{t}s^2(t-s)^{2\alpha-1}\,\mathrm{d}s \nonumber\\
		&\hspace*{3cm}+\frac{2c_{1}^2}{(\Gamma(\alpha))^{2}}\frac{1}{\alpha(\alpha+1)}\int_{0}^{t}s(t-s)^{2\alpha}\,\mathrm{d}s\nonumber\\
		&=\frac{c_2t^{2\alpha+1}}{(2\alpha+1)(\Gamma(\alpha+1))^{2}}  +\frac{\left(c_1t^{\alpha+1}\right)^2}{(\Gamma(\alpha+2))^{2}}.\label{varintgcp}
	\end{align}
	Thus, considering \eqref{meanintgcp} and \eqref{varintgcp}, we obtain the required result.
\end{proof}
\begin{remark}
	On taking $k=1$ in Theorem \ref{varmat}, we obtain the variance of fractional integral of Poisson process (see Orsingher and Polito (2013), Theorem 3.1).
\end{remark}
\begin{proposition}
	The conditional mean of the fractional integral of GCP is given by 
	{\footnotesize\begin{equation*}
		\mathbb{E}(\mathcal{M}^\alpha(t)|M(t)=n)=\frac{1}{\Gamma(\alpha)}\sum_{r=0}^{n}\frac{re^{-\Lambda t}}{\mathrm{Pr}\{M(t)=n\}}\sum_{\Omega(k,r)}\Big(\prod_{j=1}^{k}\frac{\lambda_j^{y_j}}{y_j!}\Big)\sum_{\Omega(k,n-r)}\Big(\prod_{j=1}^{k}\frac{\lambda_j^{z_j}}{z_j!}\Big)t^{y+\alpha +z}B(y+1,\alpha+z),
	\end{equation*}}
	where $y=y_1+y_2+\dots+y_k$ and $z=z_1+z_2+\dots+z_k$.
\end{proposition}
\begin{proof}
	From \eqref{fringc}, we have
	\begin{align*}
		\mathbb{E}(\mathcal{M}^\alpha&(t)|M(t)=n)\\
		&=\frac{1}{\Gamma(\alpha)}\int_{0}^{t}(t-s)^{\alpha-1}\mathbb{E}(M(s)|M(t)=n)\mathrm{d}s\\
		&=\frac{1}{\Gamma(\alpha)}\sum_{r=0}^{n}r\int_{0}^{t}(t-s)^{\alpha-1}\mathrm{Pr}\{M(s)=r|M(t)=n\}\mathrm{d}s\\
		&=\frac{1}{\Gamma(\alpha)}\sum_{r=0}^{n}\frac{re^{-\Lambda t}}{\mathrm{Pr}\{M(t)=n\}}\sum_{\Omega(k,r)}\Big(\prod_{j=1}^{k}\frac{\lambda_j^{y_j}}{y_j!}\Big)\sum_{\Omega(k,n-r)}\Big(\prod_{j=1}^{k}\frac{\lambda_j^{z_j}}{z_j!}\Big)\int_{0}^{t}s^{y}(t-s)^{\alpha+z-1}\mathrm{d}s
	\end{align*}
	which reduces to the required result.
\end{proof}
\section*{Acknowledgement}
The second author thanks Government of India for the grant of Prime Minister's Research Fellowship, ID 1003066.

\end{document}